\documentclass{article}[10pt]
\usepackage{amsmath,amssymb,amsthm,txfonts}
 \usepackage{graphicx}
 \usepackage{color}

\usepackage{tikz, framed}
\usetikzlibrary{arrows,shapes,matrix,decorations.pathmorphing}
\usetikzlibrary{backgrounds}

\theoremstyle{plain}%{\bf}{\it}
\newtheorem{Theorem}{Theorem} 

\newtheorem{Definition}[Theorem]{Definition}

\newtheorem{Lemma}[Theorem]{Lemma}
\newtheorem{Proposition}[Theorem]{Proposition}
\newtheorem{Corollary}[Theorem]{Corollary}

\theoremstyle{definition}%{\bf}{\rm}

\theoremstyle{remark}%{\it}{\rm}
\newtheorem{examplex}[Theorem]{Example}
\newenvironment{example}
  {\pushQED{\qed}\examplex}
  {\popQED\endexamplex}

\newtheorem{Remark}[Theorem]{Remark}

\newtheorem{History}[Theorem]{Historical Remark}

 \makeindex
 
 \newcommand{\NN}{\mathbb{N}}

  \newcommand{\cT}{{\sf{T}}}
 \newcommand{\cN}{{\sf{N}}}
  \newcommand{\cG}{{\sf{G}}}

 \newcommand{\mult}{\textrm{M}}
 \newcommand{\mcm}{\textrm{lcm}}
  \newcommand{\mcd}{\textrm{GCD}}
  \newcommand{\ck}{{\mathbf{k}}}
  \newcommand{\st}{{\textrm{such that\ }}}

 \newcommand{\kT}{\mathcal{T}}

\def\Forall{\mbox{\ for each }}

\def\Either{\mbox{\ either }}

\def\Or{\mbox{\ or }}
\def\Bbb#1{{\mathbb #1}}
\def\Cal#1{{\cal #1}}

\def\then{\;\Longrightarrow\;}

\def\lcm{\mathop{\rm lcm}\nolimits}

 \title{Combinatorics of involutive divisions}
 \author{M.Ceria}
 \date{}

\begin{document}
\newpage
\fussy
\maketitle
\begin{abstract}
The classical involutive division theory by Janet decomposes in the same way both the ideal and the escalier. 
The aim of this paper, following Janet's approach, is to discuss  the combinatorial properties of involutive divisions, when defined on the set  of all terms in a fixed degree $D$, postponing  the discussion of ideal membership and related test.
\\
We adapt the theory by Gerdt and Blinkov, introducing relative involutive divisions and then, 
given a complete description of the combinatorial structure of a relative 
involutive division, we  turn our attention to the problem of membership.
In order to deal with this problem, we introduce two graphs as tools; one is strictly related to Seiler's $L$-graph, whereas the second generalizes it, to cover the case of "non-continuous" (in the sense of Gerdt-Blinkov) relative involutive divisions.
Indeed, given an element in the ideal (resp. escalier), walking backwards (resp. forward) in the graph, we can identify all the other generators of the ideal (resp. elements of degree $D$ in the escalier).
\end{abstract}

\section{Introduction}\label{Intro}
Denote by $\mathcal{P}:=\mathbf{k}[x_1,...,x_n]$ the graded ring of polynomials in
$n$ variables with coefficients in the field $\bf{k}$,  $char({\bf{k}})=0$ and by
$\mathcal{T}\!:=\!\{x^{\gamma}\!:=\!x_1^{\gamma_1}\cdots
x_n^{\gamma_n} \vert \,\gamma\!=\!(\gamma_1,...,\gamma_n)\in {\mathbb{N}}^n \}$
the \emph{semigroup of terms} generated by the set $\{x_1,...,x_n\}$.

Given a monomial/semigroup ideal $J\subset\mathcal{T}$ and its minimal set of generators ${\sf G}(J)$  (also called its \emph{monomial basis}),
  Janet introduced in  \cite{J1} both the notion of \emph{multiplicative variables} and the connected decomposition of $J$ into disjoint 
  \emph{cones}. Then, he gave a procedure (\emph{completion})   
  to produce such a decomposition.

In the same paper, in order to describe Riquier's \cite{Riq} formulation 
of the description for the general solutions of a PDE problem, Janet 
gave a similar decomposition in terms of disjoint cones, 
generated by multiplicative variables, also for the related normal set/order ideal/\emph{escalier}
${\bf{N}}(J):=\mathcal{T}\setminus J.$

 Later in \cite{J2,J3,J4}, he gave a completely different decomposition (and the related algorithm for computing it) which labelled as \emph{involutive} and which is behind   both Gerdt-Blinkov 
\cite{GB1, GB2, GB3} procedure for computing Gr\"obner bases and Seiler's \cite{SeiB} theory of involutiveness.

The aim of Janet in these three papers was twofold:
\begin{enumerate}
 \item to reinterpret, in terms of multiplicative variables and cone decomposition, the solution of  PDE problems given by Cartan \cite{Car1, Car2, Car3}, whence the name \emph{inolutiveness};
 \item to re-evaluate within his theory the notion of \emph{generic initial ideal} introduced by Delassus \cite{Del1, Del2, Del3} and the correction of his mistake by Robinson \cite{Rob1, Rob2} and 
Gunther \cite{Gun1, Gun2}, who point out that the notion requires $J$ to be  Borel-fixed 
 (a modern but identical reformulation was proposed by Galligo  \cite{GAL}, who merged the considerations of Hironaka \cite{Hir} and Grauert \cite{Gra}; see also \cite{ GS} and \cite{Ei}); Janet 
remarked that all Borel-fixed ideals are involutive, but the converse is false.
\end{enumerate}

More precisely, in his survey \cite{J3} Janet presents, as \emph{nouvelle formes cannoniques}, the results by Delassus,
Robinson and Gunther and compares them with the one    
which can be deduced from an involutive basis;
and in  \cite[p.62]{J4}, assuming to have a  given a homogeneous  ideal $\mathcal{I} \subset\mathcal{P}$ within a generic frame of coordinates,  he reformulates Riquier's completion proposing 
essentially a Macaulay-like construction, iteratively computing the vector-spaces 
$\mathcal{I}_d :=\{f\in\mathcal{I} : \deg(f) = d\}$ until Cartan test grants that Castelnuovo-Mumford \cite[pg.99]{Mum} regularity $D$ has been reached. This would allow him to consider
the  monomial ideal ${\sf T}(\mathcal{I})$ of the leading terms (in the sense of Gr\"obner basis theory) w.r.t. a deg-lex term-ordering and obtain the related involutive reduction required by 
Riquier's procedure.

The results on involutiveness presented in both papers, however, simply restate the results of \cite{J2} which reinterprets Cartan's 
result in terms of multiplicative variables; more precisely Janet assumes to have a set of forms of degree $D$ which satisfies  Cartan test and directly considers both
the  monomial ideal $${\sf T}:={\sf T}(\mathcal{I})\subset \mathcal{T}_{\geq D} =:\{t\in\mathcal{T}, deg(t)\geq D\}$$
and  the partial \emph{escalier} $${\sf N} := \mathcal{T}_{\geq D}\setminus {\sf T} = \{\tau\in{\sf N}(\mathcal{I}) :deg(\tau)\geq  D\}$$ decomposing both of them in terms of disjoint cones, generated 
by multiplicative variables. Actually, he simply explicitly demotes the r\^ole of the ideal in this construction\footnote{\emph{La proposition est vraie en particuier pour le syst\`eme involutif 
constitu\'e par \emph{tout} les monomes d'ordre [$D$].}
\cite[p.46]{J2}}
  considering the whole set $\mathcal{T}_{\geq D}$ and decomposing it in terms of disjoint cones generated by multiplicative variables, the related set of vertices being the set of all monomials 
$\mathcal{T}_D =\{\tau\in \mathcal{T} : \deg(\tau)=D\}$.

\medskip

The aim of this paper is to discuss involutiveness following the approach proposed by Janet in \cite{J2}; 
in particular we postpone the discussion of ideal membership and related test only after 
having performed a deep reconsideration of the combinatorial properties of involutive divisions \cite{GB1,GB2,GB3},
when defined on the set  $\mathcal{T}_D$.

To do so, we of course apply the theory of involutive divisions, set  up by Gerdt--Binklov \cite{GB1,GB2,GB3},
but we are forced to slightly adapt it, talking about \emph{relative} 
involutive divisions, and requiring that the union of all the cones produces the ideal $\mathcal{T}_{\geq D}$ 
and that the cones are disjoint; in fact our setting considers the  \emph{single} finite set  $\mathcal{T}_D$ and thus does 
not require (as, of course, they need) comparing different divisions. 

Moreover, the  aim of their theory is to produce a setting for describing and building a Riquier-Janet procedure for computing Gr\"obner-like bases for ideals; 
thus they cannot assume neither that the division is 
involutive, \emph{i.e.} that the union of all the cones defined on a set $U$ produces the semigroup ideal generated by $U$ 
(this being in their setting the aim of the procedure) nor uniqueness of involutive 
divisors, \emph{i.e.} that
that all cones are disjoint (the failure of this condition triggering the completion procedure). 
These two conditions are instead essential to grant that (the implicit procedure is completed and that) 
a unique decomposition is available both for the given ideal (granting unique reduction) and its associated 
\emph{escalier} (granting standard Hironaka-like description of canonical forms).

We discuss the combinatorial structure of relative 
involutive divisions; we begin with the combinatorial formula given by Janet \cite{J1,J2,J3} and Gunther \cite[p.184]{Gun3} evaluating, for each $i, 1\leq i \leq n$, the number $\sigma_i$ of the cones 
having $i$ multiplicative variables and which is, essentially, an adaptation of Vandermonde's convolution \cite[pg.492]{Vdm}; next we prove a set of Lemmata,
which allows us  to sketch an approach for imposing a relative 
involutive division structure on $\mathcal{T}_{\geq D}$ and which will be generalized 
  to a procedure to list all the possible relative involutive divisions up to symmetries

We further characterize the relative 
involutive divisions which are Pommaret divisions up to a relabelling of the variables.

Thus, given a complete description of the combinatorial structure of a relative 
involutive division, we  turn our attention to the problem of membership.
Let us begin with the trivial remarks that if a term $u\in\mathcal{T}_D$ is contained (or is a generator) of the  monomial ideal ${\sf T}$, 
\begin{itemize}
\item the whole cone whose vertex is $u$ is contained in the ideal and that
\item for each non-multiplicative variable $x$, there is necessarily a term $v \in\mathcal{T}_D$, $v \neq u$ s.t. $xu$ belongs to the cone whose vertex is $v$ and that such vertex (and cone) 
necessarily belongs to ${\sf T}$;
\item conversely, if  $v$ belongs to ${\sf N}$, not only its related cone belongs to ${\sf N}$, but the same holds to $u$ and its related cone.
\end{itemize}
Moreover if ${\sf T}$ is not trivial then both
\begin{itemize}
 \item the single monomial $m$ which has no non-multiplicative variables (called ``peak'' throughout the paper), and its cone  necessarily belong to ${\sf T}$ while
 \item for at least a value $i, 1\leq i \leq n$, $x_i^D\in{\sf N}$.
\end{itemize}

On the basis of these remarks, we can define on  $\mathcal{T}_D$ a rooted directed graph whose root is $m$
  and where an arrow $u\rightarrow v$ is given when, for a non 
multiplicative variable $x_i$ for $u$, $x_iu$ belongs to the cone whose vertex is $v$. Of course such graph is redundant and our aim is to give a (more compact, non necessarily minimal) directed graph 
which has the following properties:
\begin{itemize}
 \item if a vertex $h$ is included in ${\sf T}$ and we walk against the flow towards the ``peak'' $m$, we reach all the terms in $\mathcal{T}_D$ which necessarily belong to ${\sf T}$ too; and
\item if a vertex $n$ is included in ${\sf N}$ and we follow the flow towards the ``mouths'' we reach all the terms in $\mathcal{T}_D$ which necessarily belong to ${\sf N}$ too.
\end{itemize}

We begin our investigation   giving conditions (based on the computation of $\mcm$'s),  which, for each $t\in{\sf T}$, $s\in U$, allows to deduce further elements $X(t,s)$ (actually 
the vertex of the cone containing $\mcm(s,t)$) which are necessary members of ${\sf T}$. Then we will give analogous $\mcm$-based conditions for ${\sf N}.$

Next  we specialize our investigation to Pommaret divisions for which we prove that it is sufficient to 
adapt Ufnarovsky graph \cite{U1,U2,U3,U4} to the commutative case in order to obtain a graph 
which has exactly the shape and properties described above.

Unfortunately, in general, a graph with the properties described above cannot exist; 
in fact we show an example in $n$ variables and degree $d=n-1$, in which the $d$ monomials with $n-1$ multipicative variables 
are connected together, via their single non-multiplicative variables, in a loop which walks around the ``peak'' $m$; moreover these $n$ monomials having either $n$ or $n-1$ multiplicative variables are s.t if one of them belongs to either ${\sf T}$ the same happens for all of them.
Moreover, if one of the terms with $n-1$ multiplicative variables lies in  ${\sf N}$, then all terms with  $n-1$ multiplicative variables lie in  ${\sf N}$, but they do not impose any condition on that with $n$ multiplicative variables.

Thus, in general, it is impossible to produce a graph as the Ufnarovsky-like   existing for Pommaret division and which has the required structure, simply by multiplying each monomial $t$ by its 
non-mutiplicable variables $x$ and the graph is obtained recording the cone in which $xt$ belongs.

The only way we are seeing for producing a graph with the required properties is to build the rendundant graph which can be obtained by testing the condition on $\mcm(s,t)$ 
  and extract a minimal subgraph, an approach which in general is NP-complete.
 
% \textbf{citare la scala e Apel: ?}

\section{Some general notation}\label{Notation}
Throughout this paper, in connection with monomial ideals, we mainly follow the notation of \cite{SPES}.
\\
We denote by $\mathcal{P}:=\mathbf{k}[x_1,...,x_n]$ the graded ring of polynomials in
$n$ variables with coefficients in the field $\ck$.\\
The \emph{semigroup of terms}, generated by the set $\{x_1,...,x_n\}$ is:
$$\mathcal{T}:=\{x^{\gamma}:=x_1^{\gamma_1}\cdots
x_n^{\gamma_n} \vert \,\gamma:=(\gamma_1,...,\gamma_n)\in \NN^n \}.$$
If $\tau=x_1^{\gamma_1}\cdots x_n^{\gamma_n}$, then $\deg(\tau)=\sum_{i=1}^n
\gamma_i$ is the \emph{degree} of $\tau$ and, for each $h\in \{1,...,n\}$
$\deg_h(\tau):=\gamma_h$ is the $h$-\emph{degree} of $\tau$.\\
\\
For each $d \in \NN$, $\mathcal{T}_d$ is the $d$-degree part of
$\mathcal{T}$, i.e.
$\mathcal{T}_d:=\{x^\gamma \in \kT \vert \, \deg(x^\gamma)=d\}$ 
and it is well known that $\vert \mathcal{T}_d \vert = {n+d-1 \choose d}$. For
each subset $M\subseteq \mathcal{T}$ we set $M_d=M\cap
\mathcal{T}_d$. The symbol $\mathcal{T}(d)$ denotes the degree $\leq d$ part of
$\mathcal{T}$, namely $\mathcal{T}(d)=\{x^\gamma \in \kT\vert \, \deg(x^\gamma)\leq d\}$.
\noindent For each term $\tau \in \mathcal{T}$ and $x_j \vert  \tau$, the only $\upsilon
\in \mathcal{T}$ \st $\tau=x_j\upsilon$ is called $j$-th \emph{predecessor} of
$\tau$.\\
A \emph{semigroup ordering} $<$ on $\mathcal{T}$  is  a total ordering
\st
$ \tau_1<\tau_2 \Rightarrow \tau\tau_1<\tau\tau_2,\, \forall \tau,\tau_1,\tau_2
\in \mathcal{T}$. For each semigroup ordering $<$ on $\mathcal{T}$,  we can represent a polynomial
$f\in \mathcal{P}$ as a linear combination of terms arranged w.r.t. $<$, with
coefficients in the base field $\mathbf{k}$:
$$f=\sum_{\tau \in \mathcal{T}}c(f,\tau)\tau=\sum_{i=1}^s c(f,\tau_i)\tau_i:\,
c(f,\tau_i)\in
\mathbf{k}^*,\, \tau_i\in \mathcal{T},\, \tau_1>...>\tau_s,$$ with
$\cT(f)=Lt(f):=\tau_1$   the 
\emph{leading term} of $f$, $Lc(f):=c(f,\tau_1)$ the  \emph{leading
coefficient} 
of $f$ and $tail(f):=f-c(f,\cT(f))\cT(f)$  the 
\emph{tail} of $f$.
\\
A \emph{term ordering} is a semigroup ordering \st $1$ is lower 
than every variable or, equivalently, it is a \emph{well ordering}.\\
Given a term $\tau \in \kT$, we  denote by $\min(\tau)$ the smallest 
variable $x_i$, $i \in \{1,...,n\}$, s.t. $x_i\mid \tau$ and, analogously, we  denote by $\max(\tau)$ 
the biggest  
variable appearing in $\tau$ with nonzero exponent.
\\
\smallskip

A subset $J \subseteq \kT$ is a \emph{semigroup ideal} if  $\tau \in J 
\Rightarrow \sigma\tau \in J,\, \forall \sigma \in \mathcal{T}$; a subset ${\sf N}\subseteq \mathcal{T}$ is a normal set  (or \emph{order ideal})  if
$\tau \in {\sf N} \Rightarrow \sigma \in {\sf N}\, \forall \sigma \vert \tau$.  We have that ${\sf N}\subseteq \mathcal{T}$ is an order ideal if and only if 
$\mathcal{T}\setminus {\sf N}=J$ is a semigroup ideal.
\\
\smallskip

Given a semigroup ideal $J\subset\mathcal{T}$  we define ${\sf 
N}(J):=\mathcal{T}\setminus J$. The minimal set of generators ${\sf G}(J)$ of $J$, called the \emph{monomial basis} of $J$, satisfies the conditions below
\begin{eqnarray*}
{\sf G}(J)&:=&\{\tau \in J\, \vert \, \textrm{ each predecessor of }\, \tau
\in \cN(J)\}\\
&=&\{\tau \in \mathcal{T}\, \vert \,\cN(J)\cup\{\tau\}\, \textrm{is an order ideal},
\, \tau \notin \cN(J)\}.
\end{eqnarray*}
% \begin{eqnarray*}
% &{\sf G}(J):=\{\tau \in J\, \vert \, \textrm{ each predecessor of }\, \tau 
% \in \cN(J)\}=&\\
% &=\{\tau \in \mathcal{T}\, \vert \,\cN(J)\cup\{\tau\}\, \textrm{is an order ideal}, 
% \, \tau \notin \cN(J)\}.
% \end{eqnarray*}
\noindent For all subsets $G \subset \mathcal{P}$,  $\cT\{G\}:=\{\cT(g),\, g \in  G\}$ and $\cT(G)$ is the semigroup ideal
of leading terms defined as $\cT(G):=\{\tau \cT(g),\, \tau \in \mathcal{T}, g \in G\}$. 
\\
Fixed a term order $<$, for any ideal  $I
\triangleleft \mathcal{P}$ the monomial basis of the semigroup ideal 
$\cT(I)=\cT\{I\}$ is called \emph{monomial basis}  of $I$ and denoted again by $\cG(I)$,
whereas the ideal 
$In(I):=(\cT(I))$ is called \emph{initial ideal} and the order ideal 
$\cN(I):=\kT \setminus \cT(I)$ is called \emph{Groebner escalier} of $I$. 

% The 
% \emph{border 
% set} of $I$ is defined as:
% 
% \begin{eqnarray*}
% {\sf B}(I) &:=& \{x_h\tau,\, 1 \leq h \leq n,\, \tau \in {\sf N}(I)\}\setminus {\sf N}(I)\\
% &=&\cT(I)\cap (\{1\}\cup \{x_h\tau,\, 1 \leq h \leq n,\, \tau \in {\sf N}(I)\}).
% \end{eqnarray*}
% % 
% % \begin{eqnarray*}
% % &{\sf B}(I):=\{x_h\tau,\, 1 \leq h \leq n,\, \tau \in {\sf N}(I)\}\setminus {\sf N}(I)=&\\
% % &=\cT(I)\cap (\{1\}\cup \{x_h\tau,\, 1 \leq h \leq n,\, \tau \in {\sf N}(I)\}).
% % \end{eqnarray*}

\section{Involutive divisions}\label{DivInvol}
In this section, following \cite{GB1,GB2}, we recall the main definitions and properties
 of involutive divisions. We also define, as an example of involutive division, Janet and
 Pommaret divisions; the latter will be very important in what follows, besides 
 being important for its link with generic initial ideals (see section \ref{Gin}).
\\
First of all, we recall  the definition of involutive division.
\begin{Definition}[Gerdt-Blinkov, \cite{GB1}]\label{Dinv}
An {\em involutive division}\index{division!involutive}\index{involutive!division} $L$ or {\em $L$-division}
on ${\Cal T}$ is a relation $\mid_L$ defined, for each finite set $U\subset{\Cal T}$, on the set
 $U \times {\Cal T}$ in such a way that the following holds for each $u,u_1\in U$
and $t, t_1\in{\Cal T}$
\begin{enumerate}
\renewcommand\theenumi{{\rm (\roman{enumi})}}
\item $u \mid_L t \then u \mid t$;
\item $u \mid_L u$ for each $u\in U$;
\item $u \mid_L ut, u \mid_L ut_1 \iff
u \mid_L utt_1$;
\item $u\mid_Lt, u_1\mid_Lt \then \Either 
u\mid_Lu_1 \Or u_1\mid_Lu$;
\item $u \mid_L u_1, u_1 \mid_L t \then u \mid_L t;$
\item if $V \subseteq U$ and $u\in V$ then 
$u \mid_L t$ w.r.t. $U \then u \mid_L t$ w.r.t. $V$.
\end{enumerate}

If $u \mid_L t = uw$, $u$ is called an {\em involutive divisor}
of $t$, $t$ is called an {\em involutive multiple}
of $u$ and $w$ is said to be {\em multiplicative} for
$u$.

If $u \nmid_L t = uw$, $w$ is said to be {\em non-multiplicative} for
$u$.
\qed\end{Definition}

This definition, for each set $U$ and each $u \in U$, partitions the set of variables in two subsets
\begin{itemize}
\item $M_L(U,u)$, containing the variables $x_i$
multiplicative for $u$:
$$x_i\in M_L(U,u)\iff u \mid_L u x_i;$$
\item $ NM_L(U,u)$, containing the variables $x_i$
non-multiplicative for  $u$:
$$x_i\in NM_L(U,u)\iff u \nmid_L u x_i.$$
\end{itemize}

Finally, for each involutive division $L$, each finite set $U\subset{\Cal T}$ and each
$u\in U$,
we denote by $L(u,U)$ the \emph{multiplicative set } for $u$, i.e. the 
set of all the terms $w\in{\Cal T}$
which are multiplicative for
$u$:
$$L(u,U) := \{w\in{\Cal T} : 
u \mid_L  uw \}.$$
Remark that condition (iii) implies that each 
$L(u,U)$ is completely characterized by the partition
$$\{x_1,\ldots,x_n\} =  M_L(U,u)\sqcup NM_L(U,u)$$
since
$$L(u,U) = \{
x_1^{a_1}x_2^{a_2}\cdots x_n^{a_n} :
a_i \neq 0 \then x_i\in  M_L(U,u)\}.$$

With this notation it is easy to realize that the definition of
involutive division can be formulated as follows:
\begin{Definition}[Gerdt---Blinkov]\label{InvDef} An {\em involutive division} $L$ or {\em $L$-division}
on ${\Cal T}$ is the assignement, for each finite set $U\subset{\Cal T}$
and each term $u\in U$ of a submonoid 
$L(u,U)\subset{\Cal T}$ such that the following holds for each $u,u_1\in U$
and $t, w\in{\Cal T}$
\begin{enumerate}
\renewcommand\theenumi{{\rm (\alph{enumi})}}
\item $t\in L(u,U), t_1 \mid t \then t_1\in L(u,U)$,
\item if $uL(u,U) \cap u_1L(u_1,U) \neq \emptyset$ then 
$u \in u_1L(u_1,U) $ or $u_1 \in uL(u,U) $;
\item if 
$u_1 = uw$ for some $w\in L(u,U)$, 
then $L(u_1,U) \subseteq L(u,U)$;
\item if $V \subseteq U$
then
$L(u,U) \subseteq L(u,V)$ for each $u\in V$.
\end{enumerate}
\end{Definition}
The next definitions \ref{IAI} and \ref{Linv} state some properties
that a set $U \subseteq \kT$, w.r.t. an involutive division $L$, may satisfy.

\begin{Definition}[\cite{GB1}]\label{IAI} 
A finite set $U\subset{\Cal T}$ 
is called 
\begin{itemize}
\item {\em involutively autoreduced} w.r.t. the
division $L$ or {\em $L$-autoreduced} if
it does not contain elements $L$-divisible by other elements in U.
\item
{\em involutive} w.r.t. the
division $L$ or {\em $L$-involutive} if $\forall u \in U$, $\forall w \in \kT$ $\exists v \in U$ s.t. $v\mid_L uw$.
\end{itemize}
\end{Definition}

\begin{Definition}[\cite{GB1}]\label{Linv}
A set $U$ is called  {\em locally involutive} with respect to the involutive division $L$ if $\forall u \in U$
$ \forall x_i \in NM_L(u,U)$ $\exists v \in U$ s.t. $v \mid_L ux_i$.
\end{Definition}

With the following definition, we introduce the concept of \emph{continuity} for an  involutive division $L$.

\begin{Definition}[\cite{GB1}]\label{Contin} The involutive division $L$ is called 
{\em continuous}\index{division!involutive!continuous}\index{involutive!division!continuous} if  for each finite  set $U\subset{\Cal T}$,
and each finite sequence of terms in $U$
$$w_1,\ldots,w_j,\ldots,w_J, $$
such that, for each $j <J$, there is $x_{i_j}\in NM_L(U,w_j) : 
w_{j+1} \mid_L w_j\cdot x_{i_j}$
the inequality
$w_j \neq w_i$ holds for each $j \neq i$ 
 \end{Definition}

 \begin{Proposition}[\cite{GB1}]\label{contin invol}
If an involutive division $L$ is continuous then local involutivity of any set $U$ implies its
involutivity.
 \end{Proposition}

 \begin{example}[Janet division]\label{JanetDiv}
  Let $U\subset{\Cal T}\subset \ck[x_1,...,x_n]$ 
  be a finite set of terms, $x_1<...<x_n$ and
$\tau := x_1^{\alpha_1}\ldots x_n^{\alpha_n}\in U.$
For each
$j, 1 \leq j \leq n$, the variable
$x_j$ is said to be  {\em multiplicative}  for $\tau$ w.r.t. $U$
if there is no
$\tau' :=x_1^{\beta_1}\ldots x_j^{\beta_j}x_{j+1}^{\alpha_{j+1}}\ldots 
x_n^{\alpha_n} \in U$ for which
$\beta_j>\alpha_j$.
\\
For example, if $U=\{x^2,xy,z^3\}\subset \ck[x,y,z]$, $x<y<z$, and Janet division  $J$ is defined on U, then
$M_J(x^2,U)=\{x\}$, $NM_J(x^2,U)=\{y,z\}$,  
$M_J(xy,U)=\{x,y\}$, $NM_J(xy,U)=\{z\}$,  
$M_J(z^3,U)=\{x,y,z\}$, $NM_J(z^3,U)=\emptyset$. 

 \end{example}

 \begin{example}[Pommaret division]\label{PommDiv}
  Consider the polynomial ring $\ck[x_1,...,x_n]$ and suppose 
  the variable ordered as $x_1<...<x_n$.
  \\ Given a term $t=x_j^{\alpha_j}\cdots x_n^{\alpha_n}$ with $\alpha_j>0$
  Pommaret division considers as multiplicative  the variables 
  $x_i,\, i\leq j$ and non-multiplicative the variables $x_k,\, k>j$.
  \\ Notice that for a term $s=x_n^{\alpha_n}$, $\alpha_n>0$, all the 
  variables are multiplicative. 	\\
  For example, if $U=\{x^2,xy,z^3\}\subset \ck[x,y,z]$, $x<y<z$, and Pommaret division  $P$ 
  is defined on U, then
$M_P(x^2,U)=\{x\}$, $NM_P(x^2,U)=\{y,z\}$,  
$M_P(xy,U)=\{x\}$, $NM_P(xy,U)=\{y,z\}$,  
$M_P(z^3,U)=\{x,y,z\}$, $NM_P(z^3,U)=\emptyset$.
 \end{example}
 
 \begin{Proposition}[\cite{GB1}]\label{JPprop}
  Janet and Pommaret monomial divisions are involutive and continuous.
 \end{Proposition}

 \section{Relative involutive divisions}\label{DefRID}
We define now a variation of Gerdt-Blinkov involutive division
(Definition \ref{Dinv}), calling it \emph{relative involutive division} since it
is defined specifically for a fixed set of terms $U \subset \kT$. In the next section, restricting
 to the case $U=\kT_D\subset \ck[x_1,...,x_n]$, 
 $D \in \NN$, i.e. to the set of all terms of degree $D$ in $n$ 
 variables, we will see the criteria for constructing relative involutive divisions on it.\\
Let us start from the definition of relative involutive division.
\begin{Definition}\label{RelInvDiv}
Let $U \subset \kT$ be a finite set of terms.  We say that a \emph{relative involutive division} $L$ is given on $U$ if, for each $u \in U$ a partition 
 $$\{x_1,...,x_n\} =M_L(u,U) \sqcup NM_L(u,U),$$
is given on the set of variables s.t. denoted 
$$L(u,U):=\{x_1^{a_1}\cdots x_n^{a_n} \, \vert \, a_i \neq 0 \Rightarrow x_i \in M_L(u,U)\}, $$
the following two conditions hold:
\begin{enumerate}
\item $\cT(U)=\bigcup_{u \in U} uL(u,U)$;
\item $\forall u,v \in U$, $uL(u,U) \cap vL(v,U) = \emptyset$.
\end{enumerate}
The set $M_L(u,U)$ is called \emph{(relative) multiplicative variable's set},  $NM_L(u,U)$ is called \emph{(relative) non-multiplicative variable's set}, whereas $L(u,U)$  is the set of \emph{(relative) multiplicative terms}.
Denoting by $C_L(u,U):=uL(u,U)$ the \emph{(relative) cone} of $u \in U$, conditions $1-2$ above may be also rewritten as:
\begin{enumerate}
\item[1'.] $\cT(U)=\bigcup_{u \in U} C_L(u,U)$;
\item[2'.] $\forall u,v \in U$, $C_L(u,U) \cap C_L(v,U) = \emptyset$.
\end{enumerate}
\end{Definition}
% % % % % \begin{Definition}\label{RelInvDiv}
% % % % % Consider a finite set of terms $U \subset \kT$. We say that a \emph{relative involutive division} $L$ is given on $U$ if, for each $u \in U$ a partition on the variables
% % % % %  $$\{x_1,...,x_n\} =M_L(u,U) \sqcup NM_L(u,U),$$
% % % % %  where $M_L(u,U)$ is called \emph{(relative) multiplicative variable's set} and  $NM_L(u,U)$ is called \emph{(relative) non-multiplicative variable's set}  is given, such that, if
% % % % % $$L(u,U):=\{x_1^{a_1}\cdots x_n^{a_n} \, \vert \, a_i \neq 0 \Rightarrow x_i \in M_L(u,U)\} $$
% % % % % is the set of \emph{(relative) multiplicative terms} the following two conditions hold:
% % % % % \begin{enumerate}
% % % % % \item $\cT(U)=\bigcup_{u \in U} uL(u,U)$;
% % % % % \item $\forall u,v \in U$, $uL(u,U) \cap vL(v,U) = \emptyset$.
% % % % % \end{enumerate}
% % % % % We denote $C_L(u,U):=uL(u,U)$ the \emph{(relative) cone} of $u \in U$ so that conditions $1-2$ above may be also rewritten as:
% % % % % \begin{enumerate}
% % % % % \item[1'.] $\cT(U)=\bigcup_{u \in U} C_L(u,U)$;
% % % % % \item[2'.] $\forall u,v \in U$, $C_L(u,U) \cap C_L(v,U) = \emptyset$.
% % % % % \end{enumerate}
% % % % % \end{Definition}
We will write $u\vert_L w$ if $w=uv$ and $v \in L(u,U)$ (so that $w \in C_L(u,U)$) and we will say that $u$ is a \emph{(relative) involutive divisor} of $w$ and that $w$ is a   \emph{(relative) 
involutive multiple} of $u$.
\\
\smallskip

In order to give a clear comparison between Gerdt-Blinkov involutive division and our relative involutive divisions, we prove the following Lemma.
\begin{Lemma}\label{RelWGerdt}
A relative involutive division $L$ on a finite set  $U \subset \kT$ satisfies conditions $(i)-(iii)$ of Definition \ref{Dinv}.
\end{Lemma}
\begin{proof}
Condition (i) is true by definition, whereas condition (ii) comes trivially observing that $\forall u \in U$, $1 \in L(u,U)$.\\
We prove now (iii) Suppose that $u \vert_L uv$ and  $u \vert_L uw$; $v=x_1^{a_1}\cdots x_n^{a_n}$,  $w=x_1^{b_1}\cdots x_n^{b_n}$, where  each variable appearing in $v$ and $w$ with nonzero exponent 
is 
multiplicative for $u$.
Take now $vw=x_1^{a_1+b_1}\cdots x_n^{a_n+b_n}$; if for some $1 \leq i \leq n$ it holds $a_i+b_i \neq 0$ 
then either $a_i$ or $b_i$ (or both) are different from zero. This implies $x_i \in M_L(u,U)$ 
and so $u \vert_L uvw$.\\
Viceversa, suppose $u \vert_L uvw$; $vw=x_1^{a_1}\cdots x_n^{a_n}$, with, for $1 \leq i \leq n$, $a_i \neq 0 \Rightarrow x_i \in M_L(u,U)$, $v=x_1^{b_1}\cdots x_n^{b_n}$, $w=x_1^{c_1}\cdots 
x_n^{c_n}$, $b_i+c_i=a_i$ for $i=1,...,n$.\\
If, for some $1 \leq i \leq n$, $b_i \neq 0$, then $a_i \neq 0$, so $x_i \in M_L(u,U)$ and the same holds for the exponents of the variables appearing in $w$, so $v,w \in L(u,U)$ and 3. is proved.
\end{proof}
As regards conditions (iv)-(v), we see that they trivially hold since their hypothesis can never happen, 
because we have imposed the relative cones to be disjoint. 
Moreover, condition (vi) does not make sense in our context, due to the 
relativity of our involutive division. Indeed, in \cite{GB1,GB2}, Gerdt and Blinkov define involutive divisions on $\kT$
as ``rules'' to be applied to any $U \subset \kT$, whereas relative involutive divisions
only involve a specific $U \subset \kT$.
% % % % % % \begin{Remark}\label{induced}
% % % % % % For completeness' sake, we remark that, given a set $U\subseteq \kT$ and a relative 
% % % % % % involutive division $L$ on $U$, we could decide to introduce the concept 
% % % % % % of relative involutive division \emph{induced} on a subset $V\subset U$, by assigning to
% % % % % %  each $t\in V$ the same multiplicative variables as the same element has in $U$. In this 
% % % % % %  case, condition (vi) holds trivially.
% % % % % % \end{Remark}

\begin{Remark}\label{invol_autored}
Given a relative involutive division $L$ on a finite set $U \subset \kT$, we can notice
that $U$ turns out to be involutively autoreduced  and $L$-involutive according to Gerdt-Blinkov definitions \cite{GB1,GB2}; this trivially follows from conditions 1-2 of Definition \ref{RelInvDiv}.
\end{Remark}

\begin{Remark}\label{NonCont}
We remark that the relative involutive divisions we are defining are not 
continuous in the sense defined by
 Gerdt-Blinkov in \cite{GB1} (see Proposition \ref{contin invol}). In our set $U$, in fact, local involutivity does not imply involutivity,
 i.e. it is not 
 true in general that 
 
  (a) $\forall u \in U$, $\forall x_j\in NM(u,U)$, $1\leq j\leq n$, $\exists v \in U$ s.t. 
     $u x_j \in vL(v,U)$

     implies that 
     
     (b) $\cT(U)=\bigcup_{u \in U} uL(u,U)$.
  \\   
 Take for example  (see also \cite{GB1})   $D=1$, $n=3$, so $U:=\mathcal{T}_1=\{x,y,z\}$ and suppose that
$\mult(x,U)=\{x,y\}$, $\mult(y,U)=\{y,z\}$, $\mult(z,U)=\{x,z\}$.
\\
This way,   $xz \in zL(z,U)$, $xy \in xL(x,U)$ and  $y z \in yL(y,U)$, so (a)
is trivially satisfied.
\\
We notice that (b) does not hold, since  
$xyz \notin xL(x,U)\cup yL(y,U)\cup zL(z,U)$, since $z\notin M(x,U)$, 
$x\notin M(y,U)$ and   $y\notin M(z,U)$. 
\\
This example also shows that the \emph{completion} procedure by Janet \emph{does not work} for a 
general relative involutive division. Indeed, the set $U$ is complete according to
 Janet definition, but 
$\cT(U)\neq \bigcup_{u \in U} uL(u,U)$.
\\
Moreover, we notice that if $\forall u,v \in U$ $\exists w \in U$ s.t. $\mcm(u,v) \in wL(w,U)$, this in general
does not imply that $\cT(U)=\bigcup_{u \in U} uL(u,U)$.
\end{Remark}

\begin{example}[Janet relative involutive division]\label{JanetInvol}
Consider a set $U\subseteq \kT \subset \ck[x_1,...,x_n]$  and define on it Janet involutive division $J$,
as in example \ref{JanetDiv}, with $x_1<...<x_n$. We know (see \cite{GB1}) that condition 2. of 
Definition \ref{RelInvDiv}
 is always satisfied, whereas condition 1. is satisfied only if $U$ is involutive
  or, in Janet's language \cite{J1}, \emph{complete}. 
\end{example}

\begin{example}[Pommaret relative involutive division]\label{PomInvol}
Consider a set $U\subseteq \kT \subset \ck[x_1,...,x_n]$  and define on it Pommaret involutive division $P$,
as in example \ref{PommDiv}. Neither condition 1. nor condition 2. are authomatically
satisfied for an arbitrarily chosen set $U\subset \kT$, as shown in the following examples:
\begin{itemize}
 \item[1.] if $U=\{x_1,x_2\}\subset \ck[x_1,x_2,x_3]$, $x_1<x_2<x_3$, then $M_P(x_1,U)=\{x_1\}$,
 $NM_P(x_1,U)=\{x_2,x_3\}$, 
 $M_P(x_2,U)=\{x_1,x_2\}$, $NM_P(x_2,U)=\{x_3\}$, so $x_1x_3 \in \cT(U)$ but it does not belong to
 $x_1L(x_1,U)\cup x_2 L(x_2,U)$.
 \item[2.] $U=\{x_1,x_1^2\}\subset \ck[x_1,x_2]$,  $x_1<x_2$, then $M_P(x_1,U)=\{x_1\}$,
 $NM_P(x_1,U)=\{x_2\}$, 
 $M_P(x_1^2,U)=\{x_1\}$, $NM_P(x_1^2,U)=\{x_2\}$, so $x_1^2 \in x_1L(x_1,U)\cap x_1^2L(x_1,U)$.
\end{itemize}
\end{example}

\section{Constructing a relative involutive division}\label{Disj}

% 
%  \textbf{Teo: mi avevi detto che non ti piaceva, ma non trovo scritto come la volevi: me lo scrivi?}\\
In this section, we specialize to the case $U=\kT_D$ and we give precise criteria to assign relative involutive divisions on $U$ in all possible ways.\\
First of all, we remark that, in order to satisfy both conditions of Definition \ref{RelInvDiv},   
$\forall d \in {\mathbb{N}}$, $\vert {\mathcal{T}}_{D+d} \vert = {D+d+n-1 \choose n-1}  $ must equal the sum of the terms 
generated by each element of 
$U:=\mathcal{T}_D$.  Given a term $u \in U$, s.t. $\vert M_L(u,U) \vert = k$ the number of terms 
in degree $d+D$, generated by $u$, multiplying it only by its multiplicative variables is 
$m_{d,k}:={d+k-1\choose  k-1}$, so we need preliminarly to partition $U$ into $n$ sets $U_i$, $1 \leq k \leq n$, 
each consisting of $a_k := \vert U_k\vert$,  elements
so that 
\begin{equation}\label{Hil}
\sum_{k=1}^{n} a_k m_{d,k} = {D+d+n-1 \choose n-1}, \Forall d\geq 0
\end{equation}

and then decide %how and
 which $k$ multiplicative variables associate to each term $u\in U_k$ 
belonging to each subset $U_k$, in such a way that the related cones do not intersect while 
covering the whole ideal $\mathcal{T}_{\geq D}$.
% \end{Remark}

\begin{History} The notation $\sum_{i=0}^{d} \chi_i {R \choose i}$ was used by Hilbert (with the sum running from $i=0$ to $i=d$) 
\cite[Th.IV,pg.512]{Hilbert} 
as a notation for his \emph{characteristiche Function}; Janet connects it with (\ref{Hil})
setting 
$$a_k := \sigma_k := \#\{\tau\in \mathcal{T}_D : \min(\tau)=k\} = {D+n-1-k \choose n-k}$$
 while describing his decomposition on terms of Cartan's results\footnote{Janet is actually defining Pommaret division.}.

It is then sufficient to note the each term $\tau= x_1^{\gamma_1}\cdots
x_n^{\gamma_n}, \;\sum_i \gamma_i \geq D$ can be uniquely decomposed as
$$\tau= \upsilon\omega, \upsilon = x_1^{\gamma_1}\cdots
x_j^{\gamma_j-\bar\gamma_j}, \omega = x_j^{\bar\gamma_j}\cdots
x_n^{\gamma_n}\in U,\; \bar\gamma_j+\sum_{i=j+1}^n \gamma_i = D$$
to elementary deduce the relation 
\begin{equation}\label{Binom}
 {D+d+n-1 \choose n-1}=\sum_{k=1}^{n}{D+n-1-k \choose n-k}{d+k-1\choose  k-1},
\end{equation}
but this would be an historical cheating since the formula was well-known in the circle of Hilbert followers (see 
for instance \cite[p.184]{Gun3}).
\end{History}

In the next proposition we give a direct 
proof of equation (\ref{Binom}).

\begin{Proposition}\label{Bin1}
With the above notation, for $a_k={D+n-1-k \choose n-k}$, $1 \leq k \leq n$, we get a partition of $\vert \kT_{D+d} \vert$. More precisely
(\ref{Binom}) holds.
\end{Proposition}
\begin{proof}
We set $N=n-1$ and reformulate (\ref{Binom}) as 
$ {D+d+N \choose N}=\sum_{k=0}^{N}{D+N-1-k \choose N-k}{d+k\choose  k}$.

Applying twice upper negation formula ${l \choose k} = (-1)^k {k-l-1 \choose k}$ and the
(generalized) Vandermonde's convolution $\sum_k {r \choose k} {s \choose n-k} = {r+s \choose n}$
to the summation on the right term of the above formula we get 
\begin{eqnarray*}
{N+d+D\choose N}
      &=&
    (-1)^N{-d-D-1 \choose N}
    \\ &=& 
    (-1)^{N}\sum_{k=0}^{N}   {-D \choose N-k}{-d-1\choose  k}
    \\ &=&
    \sum_{k=0}^{N}  (-1)^{N-k}  {-D \choose N-k} (-1)^k{-d-1\choose  k}
     \\ &=&
    \sum_{k=0}^{N}{D+N-1-k \choose N-k}{d+k\choose  k}, 
    \end{eqnarray*}
    proving the assertion.
 \end{proof}
% % %  \begin{Remark}\label{Gunther}
% % %   The formula (\ref{Binom})  is an obvious variant of Gunther's formula:
% % %  $${t+m-1 \choose m-1}=\sum_{l=1}^{m}{t-n+k+l-1 \choose l-1}{n-k-1+m-l\choose  m-l}$$
% % %  
% % %  \end{Remark}

Now we have to prove that $a_k={D+n-1-k \choose n-k}$,$1 \leq k \leq n$, is the unique choice leading to a partition of the form defined above

% 
% \begin{equation}
%  {D+d+n-1 \choose n-1}=\sum_{k=1}^{n}{D+n-1-k \choose n-k}{d+k-1\choose  k-1},
% \end{equation}
% where   ${D+d+n-1 \choose n-1}$ represents the number of terms in $n$ variables and 
% degree $d+D$, ${d+k-1\choose  k-1}$ represents the number of terms in $\kT_{d+D}$,
%  generatd by a term in $\kT_D$ using $k$ multiplicative variables and 
% ${D+n-1-k \choose n-k}$ is the number of terms with $k$ multiplicative variables.
% 

\begin{Proposition}\label{UniqBinom}
 The decomposition of equation (\ref{Binom}) is unique. 
\end{Proposition}
\begin{proof} 
% DEVO PRIMA METTERE CHE c'e' quello con n variabili per forza in modo
% da poterlo isolare nella somma.\\
For each $l \in \NN$, we first define the polynomial $f_l(x):={x+l \choose l}$, 
which is trivially monic of degree $l$ in $x$.\\
It is clear that ${D+d+n-1 \choose n-1}=f_{n-1}(d+D)$
and 
that ${d+k-1\choose  k-1}=f_{k-1}(d)$, so we may rewrite (\ref{Binom}) as 
$$f_{n-1}(d+D)= \sum_{k=1}^{n}{D+n-1-k \choose n-k} f_{k-1}(d)$$
$$f_{n-1}(d+D)-f_{n-1}(d)= \sum_{k=1}^{n-1}{D+n-1-k \choose n-k} f_{k-1}(d).$$
Since $d$ does not appear in the binomial coefficients ${D+n-1-k \choose n-k}$, we may
 see them as integer coefficients, so we may write 
$$f_{n-1}(d+D)-f_{n-1}(d)= \sum_{k=1}^{n-1}a_k f_{k-1}(d),$$
 where $a_k={D+n-1-k \choose n-k},\, k=1,...,n-1$.\\
 Suppose now there is another decomposition, so that 
 $$f_{n-1}(d+D)-f_{n-1}(d)= \sum_{k=1}^{n-1}b_k f_{k-1}(d),$$
for some $b_k \in \NN, \, k=1,...,n-1$, so that 
\begin{equation}\label{UB}
 \sum_{k=1}^{n-1}(a_k-b_k) f_{k-1}(d)=0
\end{equation}
and consider the maximal $k_0$ s.t.
$a_{k_0}\neq b_{k_0}$. We have $\cT (f_{k-1}(d))=d^{k-1}$, which appears 
only once in (\ref{UB}), with coefficient $(a_{k_0}-b_{k_0})\neq 0$, so 
$\sum_{k=1}^{n-1}(a_k-b_k) f_{k-1}(d)$ cannot be zero, leading to a contradiction.
\end{proof}

The propositions above show that if a relative involutive division for $\kT_D$ exists, 
multiplicative variables must be assigned to terms in a way making the 
condition (\ref{Binom}) satisfied.
\\
Anyway, we still have not proved that such a decomposition can be achieved, 
for each $D\in \NN$. 
Actually, the answer is positive, as we can see in the following example

\begin{example}[Pommaret division -- relative case for $T=\kT_D$.]
Consider $U=\kT_D$, $D \in \NN$ and define Pommaret division on $U$, with $x_1<x_2<...<x_n$.
\\
Clearly, $\forall t \in U$ $t \in tL(t,U)$ and $\forall s \in U$, $s \neq t$, $t \notin sL(s,U)$.
\\
Take now $t_1=x_1^{\alpha_1} \cdots x_n^{\alpha_n}\in \kT_{>D}$ i.e. a term in $\kT$ of 
degree strictly greater than $D$.
Let 
$$s_1:=x_i^{\beta_i}x_{i+1}^{\alpha_{i+1}} \cdots x_n^{\alpha_n}, 1\leq i\leq n, \, \beta_i\leq \alpha_i, \,
D=\beta_i+\sum_{j=i+1}^n \alpha_j.$$

Clearly $s_1\in U$ is the only element of $U$ which Pommaret-divides $t_1$;
thus $\forall t_1 \in \kT_{\geq D}$,
$\exists ! s_1 \in U$ s.t. $t_1 \in s_1L(s_1,U)$, i.e.   for $U=\kT_D$, Pommaret division
is a relative involutive division.
% % % % Take now $t_1=x_1^{\alpha_1} \cdots x_n^{\alpha_n}\in \kT_{>D}$ i.e. a term in $\kT$ of 
% % % % degree strictly greater than $D$.
% % % % Let $s_1$ the degree $D$ term obtained taking from $t_1$  the product of the biggest vaiables
% % % % of $t_1$ with exponent smaller or equal to that of the same variable in $t_1$. 
% % % % Then $s_1 \in U$ and $t_1 \in s_1L(s_1,U)$; moreover, since we have defined Pommaret division, 
% % % % the multiplicative terms of any $u \in U$ are formed only by variables smaller than $\min(u)$,
% % % % it is not possible that $t_1$ belongs to another set of the form $uL(u,U)$ for some $u\in U$.\\
% % % % We have  shown that $\forall t_1 \in \kT_{\geq D}$,
% % % % $\exists ! s_1 \in U$ s.t. $t_1 \in s_1L(s_1,U)$, i.e. that, for $U=\kT_D$, Pommaret division
% % % % is a relative involutive division.
\\
The terms in $U$ with minimal variable $x_k$  $1\leq k \leq n$ are exactly the terms of degree $D$ 
in $x_k,...,x_n$ which contain $x_k$, so their number is the difference between the terms of 
degree $D$ 
in $x_k,...,x_n$ and those in  $x_{k+1},...,x_n$ of the same degree:
$${D+(n-k+1)-1 \choose (n-k+1)-1}- {D+(n-k+1)-2 \choose (n-k+1)-2}=$$
$${D+(n-k+1)-2 \choose (n-k+1)-2}+{D+(n-k+1)-2 \choose (n-k+1)-1} - {D+(n-k+1)-2 \choose (n-k+1)-2}= $$
$$ {D+(n-k+1)-2 \choose (n-k+1)-1}={D+n-k-1 \choose n-k}.$$
These terms have exactly $k$ multiplicative variables, so decomposition of 
equation (\ref{Binom}) holds for Pommaret division.
\\
In this particular case, i.e. for $U=\kT_D$, $D \in \NN$, we can also notice 
that Pommaret division and Janet division coincide as relative involutive divisions, with the same variable ordering.
\\
Indeed, let $t=x_i^{\alpha_i}\cdots x_n^{\alpha_n} \in U $ and suppose $\min(t)=x_i$, so $\alpha_i>0$, for some $1\leq i \leq n $.
\\
If $x_j>x_i$ then $\frac{tx_j}{x_i}\in U$, so by definition of Janet division, $x_j$ cannot 
be multiplicative for $t$.\\

On the other hands, if $x_j \leq x_i$, for $x_j$ not being multiplicative, we should find 
in $U$ a term $x_1^{\beta_1}  \cdots x_n^{\beta_n}$ with
\begin{itemize}
\item $\beta_l = \alpha_{l}, l\geq i,$ and $\beta_j> {\alpha_{j}}=0$ if $x_j<x_i$,
\item $\beta_l = \alpha_{l}, l> i,$ and $ \beta_j\geq{\alpha_{j}}+1$ if $x_j=x_i$.
\end{itemize}

Due to degree reasons, this is clearly impossible for a term in $U$, so each variable 
$x_j\leq x_i$ must be multiplicative for $t$.\\
\medskip 

% % % On the other hands, if $x_j \leq x_i$, for $x_j$ not being multiplicative, we should find 
% % % in $U$ a term with the same exponent of $t$ of the variables $\geq x_i$ ($>x_i$ if $x_j=x_i$),
% % % and exponent of $x_j$ bigger than that of $x_j$ in $t$ (so, exponent $1$ if $x_j<x_i$  and
% % % $\alpha_i +1$ if $x_j=x_i$). This is clearly impossible for a term in $U$, so each variable 
% % % $x_j\leq x_i$ must be multiplicative for $t$.
As an example, consider the terms in $n=3$ variables and degree $D=2$, supposing $x<y<z$; Pommaret division
is defined 
\begin{center}
\begin{tabular}{|c|c|}
\hline
Terms & Multiplicative Variables\\
\hline
$x^2$& $x$\\
\hline
$xy$& $x$\\
\hline
$y^2$ &$x,y$\\
\hline
$xz$ & $x$\\
\hline
$yz$ & $x,y$\\
\hline
$z^2$ & $x,y,z$\\
\hline
\end{tabular}
\end{center}
 We have exactly ${3 \choose 2} = 3$ terms with only one multiplicative variable (namely $x^2,xy,xz $),
  ${2 \choose 1} = 2$ terms with two multiplicative variables (namely $y^2,yz $) and 
  ${1\choose 0}=1$ term with thre multiplicative variables (namely $z^2$).
 \\
 If we define Janet division on the same set we get exactly the same partition into multiplicative
  and non-multiplicative variables as for Pommaret division. Indeed:
  \begin{itemize}
   \item $x^2$: $x$ is multiplicative for $x^2$; since $xy \in U$ $y$ is not 
   multiplicative and the same goes for $z$, being $xz \in U$.
  
   \item $xy$: $x$ is multiplicative for $xy$; since $y^2 \in U$ $y$ is not 
   multiplicative and the same goes for $z$, being $xz \in U$.
  
  \item $y^2$: $x,y$ are multiplicative for $y^2$; since $xz \in U$ $z$ is not 
   multiplicative.
  
    \item $xz$: $x$ is multiplicative for $xz$; since $yz \in U$ $y$ is not 
   multiplicative and the same goes for $z$, being $z^2 \in U$.

  \item $yz$: $x,y$ are multiplicative for $yz$; since $z^2 \in U$ $z$ is not 
   multiplicative.
   
   \item $z^2$: $x,y,z$ are all multiplicative for $z^2$.
  
  \end{itemize}

 We finally point out that, in this example, the multiplicative sets are all contained one in another:
 $$\{x\} \subset \{x,y\} \subseteq \{x,y,z\}.$$
We will come back to this fact in section  \ref{PRID}, while focusing on 
the main properties of Pommaret relative involutive division.

\end{example}

% With the above example, we have proved the existence of such a decomposition. 
% 
% With the above example, we show that Pommaret division satisfies (\ref{Binom}).

In order to give all the possible decompositions of $\kT_D$ in cones, we have now to study the criteria for choosing 
multiplicative variables to assign to each term in the set $U=\kT_D$, so that conditions 1., 2. of Definition \ref{RelInvDiv} are satisfied.

First of all, we prove a simple property of pure powers.

\begin{Lemma}\label{PotPura}
 Let $U:=\kT_D$, $D \in \NN$ and suppose that $\forall u \in U$ a partition of the variables 
 into multiplicative and non-multiplicative ones
 $\{x_1,...,x_n\} =M(u,U) \sqcup NM(u,U)$
 is given s.t., with the above notation, $\cT(U)=\kT_{\geq D}=\bigcup_{u \in U} uL(u,U)$. Then
 $\forall 1\leq i\leq n $, $x_i \in M(x_i^D,U).$
\end{Lemma}
\begin{proof}
Consider a term $x_i^h, h>D$
By assumption there is a term $u\in U$ s.t $x_i^h=uv$ $h>D$  
with $v\in L(u,U)$.
The only solution is $u=x_i^D$, $v=x_i^{h-D}$ which implies $x_i$ multiplicative for $x_i^D$.
% %  Suppose by contradiction that $x_i \notin M(x_i^D,U).$
% % \\
% % Let $h \in \NN$, $h>D$ and consider the term $x_i^h$; since 
% %  $\cT(U)=\kT_{\geq D}=\bigcup_{u \in U} uL(u,U)$, there must be an element $u \in U$ s.t.
% %  $x_i^h \in uL(u,U)$. This is impossible, since $\exists ! u \in U$ s.t. $u \mid x_i^h$, namely 
% %  $u=x_i^D$, and $x_i^h \in uL(u,U)$ would imply $x_i^{h-D} \in L(u,U)$ and so  $x_i \in M(x_i^D,U)$, contradicting
% %   our assumption.
\end{proof}

The following proposition gives a direct proof of the obvious fact that one and only one term $t \in \kT_D$ must have $M_L(t,U)=\{x_1,...,x_n\}$.

\begin{Proposition}\label{UnoTutteVar}

 Let $U:=\kT_D$, $D \in \NN$ and suppose that $\forall u \in U$ a partition of the variables 
 into multiplicative and non-multiplicative ones
 $\{x_1,...,x_n\} =M(u,U) \sqcup NM(u,U)$
 is given s.t., with the above notation, $\cT(U)=\kT_{\geq D}=\bigcup_{u \in U} uL(u,U)$. Then
 $\exists t \in U$ s.t.
 $M(t,U)= \{x_1,...,x_n\}$.
\end{Proposition}
\begin{proof}
Consider the term   $v:=\mcm(u\, \vert \, u \in U)=x_1^D\cdots x_n^D\in \mathcal{T}_{\geq D}$.

Since $\cT(U)=\kT_{\geq D}=\bigcup_{u \in U} uL(u,U)$ there must be a term $t \in U$ s.t. $v \in tL(t,U)$:
 \begin{itemize}
  \item[a)] if $t=x_i^D$, for some $i \in \{1,...,n\}$, then $v=t \prod_{j\in\{1,...,n \}\setminus \{i\}  }x_j^D$,
  so, $\forall j\in\{1,...,n \}\setminus \{i\}$, $x_j \in M(t,U)$. Since, by the above Lemma \ref{PotPura}, $x_i \in M(t,U)$,
  we can conclude that $M(t,U)= \{x_1,...,x_n\}$;
  \item[b)] if   $t=x_1^{\beta_1}\cdots x_n^{\beta_n}$, with $\sum_{i}\beta_i=D$, s.t. $\exists i_1,i_2 \in \{1,...,n\}$,
  $i_1\neq i_2$ s.t. 
  $\beta_{i_1}\neq 0$ and  $\beta_{i_2}\neq 0$, then  $\beta_i <D, \,\forall i \in \{1,...,n\}$.
  \\ In this case, $v= t x_1^{D-\beta_1}\cdots x_n^{D-\beta_n}$, $D-\beta_i >0$, for each $i \in \{1,...,n\}$ and this implies
  $M(t,U)= \{x_1,...,x_n\}$.  
 \end{itemize}
\end{proof}

The only term $t$ such that $M(t,U)=\{x_1,...,x_n\} $ is called the \emph{peak} of $U$.

\begin{Remark}\label{NotoJanet}
 Note that the Janet-Gunther formula (\ref{Binom})  already told us that there is exactly a \emph{single} such $t$.
\end{Remark}

\begin{Remark}\label{Numerologia}
Consider again the setting of Remark \ref{NonCont}.\\
The non-continuity of the described set can be seen also by observing that there are $10$ terms in three variables and 
degree three. Each term of  $\mathcal{T}_1$ has $2$ multiplicative variables, so it generates
 only   $3$ different terms in degree $3$ (and there is no intersection between $xL(x,U)$, $yL(y,U)$ and $zL(z,U)$) 
so, in degree $3$, we only get $9$ terms, instead of $10$  ($xyz$ is exactly the missing one), so this is due to the fact that formula (\ref{Binom}) does not hold, since there is no peak in $U$.
\end{Remark}

In the remanining part of this section we prove a criterion for defining a partitio in multiplicative and non-multiplicative variables s.t. the resulting cones are disjoint.

\begin{Lemma}\label{lcmDentro}
Let $U:=\kT_D$, $D \in \NN$ and suppose that $\forall u \in U$ a partition of the variables 
 into multiplicative and non-multiplicative ones
 $\{x_1,...,x_n\} =M(u,U) \sqcup NM(u,U)$
 is given.\\
 Let $u,v \in U$ and $w=uw_1=vw_2=\mcm(u,v)$. Then
 $$uL(u,U)\cap vL(v,U) \neq \emptyset \Leftrightarrow  w \in uL(u,U)\cap vL(v,U).$$
\end{Lemma}
\begin{proof}
We prove only the non-trivial part of the statement. \\
Suppose $t\in uL(u,U)\cap vL(v,U)$, so $t=ut_1=vt_2$ with $t_1\in L(u,U)$ and  $t_2\in L(v,U)$.\\
 Since $w \mid t$ we have $t =wm$ for some $m \in \kT$, so 
 
 $$ t=  ut_1=wm=uw_1m. $$
%  Dividing by $u$ we get $t_1=w_1m \in L(u,U)$, so
% % %  , as proved in Lemma
% % %  \ref{RelWGerdt}
%  $w  \in L(u,U)$.
% \\ 
%  With an analogous argument for $v$ we can show that  $w_2  \in L(v,U)$,
%  
  Dividing by $u$ we get $t_1=w_1m \in L(u,U)$ whence $w_1\in L(u,U)$ and $ w\in uL(u,U)$.
With an analogous argument for v we can show that $ w\in vL(v,U)$,  concluding our proof.
 
\end{proof}

\begin{Proposition}\label{Inibizione1}
Let $U=\kT_D$, $D \in \NN$ and fix on $U$ a relative involutive division $L$.
\\Let $t \in U$ s.t. $M_L(t,U)=\{x_1,...,x_n\}$ and $A:=\{x_{j_1},...,x_{j_h}\}\subset \{x_1,...,x_n\} $ 
be the set of all and only the variables appearing in $t$ with nonzero exponent.
 Then, $\forall u \in U\setminus \{t\}$, $A\nsubseteq M_L(u,U)$.
\end{Proposition}
\begin{proof}
 Suppose   $A \subseteq M_L(u,U)$  and let 
 $w=\mcm(t,u)=\frac{ut}{\mcd(t,u)}$; it may be regarded as
 \begin{itemize}
  \item[a)]  $w=u\frac{t}{\mcd(t,u)}$: the variables appearing in 
    $\frac{t}{\mcd(t,u)}$ with nonzero exponent, all belong to $A\subseteq M_L(u,U)$, so  $w\in C_L(u,U)$;
 \item[b)] $w=t\frac{u}{\mcd(t,u)}$: since $M_L(t,U)=\{x_1,...,x_n\}$, $w\in C_L(t,U)$.
 \end{itemize}
 Thus we can conclude, since   $C_L(u,U)\cap C_L(t,U)\neq \emptyset$ and so condition (ii) of 
 the definition of relative involutive division is contradicted.
\end{proof}
It is now obvious the following
\begin{Corollary}\label{No2CnTutte}
Let $U=\kT_D$, $D \in \NN$ and fix on $U$ a relative involutive division $L$.
Then, there exists one and only one $t\in U$ s.t. $M_L(t,U)=\{x_1,...,x_n\}$.
\end{Corollary}
\begin{Remark}\label{NoInfastidite}
 Let $t \in U$ be the only term s.t. $M(t,U)=\{x_1,...,x_n\}$ and suppose,
 given $i \in \{1,...,n\}$,
 that $t \neq x_i^D$. In this case, $\forall h \geq 0$ we have 
 $x_i^{D+h}\notin tL(t,U)$. Indeed, since $deg(t)=D$ and $t \neq x_i^D$,
 there exists $j \in \{1,...,n\}\setminus\{i\}$ s.t. $x_j \mid t$ and 
 this trivially implies that $\forall t \in tL(t,U)$, $x_j \mid t$ so 
  $x_i^{D+h}\notin tL(t,U)$.
  
This implies that, whatever is the term   $t \in U$ 
s.t. $M(t,U)=\{x_1,...,x_n\}$, the limitations on the choices of multiplicative
 variables imposed by $t$ cannot affect the (necessary) choice imposed 
 by Lemma \ref{PotPura}.\\
Similarly, we can observe that the assignment $x_i \in M(x_i^D,U)$, forall
 $1\leq i \leq n$ does not impose any condition on the future choices, since 
 $\forall h \geq 0$, $x_i^{D+h}$ is not multiple of any $u \in U\setminus\{x_i^D\}$.
\end{Remark}
We see now a criterion for setting a partition on the variabiles into multiplicative and
non-multiplicative ones, so that no intersection between sets of the form $t L(t,U)$ may
arise.

\begin{Proposition}\label{Inibizione2}
 Let $U:=\kT_D$, $D \in \NN$ and suppose that $\forall u \in U$ a partition of the variables 
 into multiplicative and non-multiplicative ones
 $\{x_1,...,x_n\} =M(u,U) \sqcup NM(u,U)$
 is given.
Let $u,v \in U$,   $w:=\mcd(u,v) $  and suppose 
 \begin{itemize}
  \item $\frac{u}{w}=x_{j_1}^{\alpha_1}\cdots x_{j_l}^{\alpha_l}$, $\alpha_i >0$, $i=1,...,l$;
  \item $\frac{v}{w}=x_{k_1}^{\beta_1}\cdots x_{k_s}^{\beta_s}$, $\beta_j >0$, $j=1,...,s$.
 \end{itemize}
Then
$M(u,U)\supseteq \{x_{k_1},...,x_{k_s}\}$ and $M(v,U)\supseteq \{x_{j_1},...,x_{j_l}\}$ if and only if  
$uL(u,U)\cap vL(v,U)\neq \emptyset$.
\end{Proposition}
\begin{proof}
Suppose $M(u,U)\supseteq \{x_{k_1},...,x_{k_s}\}$ and $M(v,U)\supseteq \{x_{j_1},...,x_{j_l}\}$ and 
let  $w':=\mcm(u,v)=\frac{uv}{w}$; then 
 \begin{itemize}
  \item[$a.$] $w'=v \frac{u}{w}=vx_{j_1}^{\alpha_1}\cdots x_{j_l}^{\alpha_l} \in vL(v,U)$, being
   $x_{j_1},...,x_{j_l}\in M(v,U)$ by hypothesis;
  \item[$b.$] $w'=u \frac{v}{w}=u x_{k_1}^{\beta_1}\cdots x_{k_s}^{\beta_s}\in uL(u,U)$, being  
  $x_{k_1},...,x_{k_s}\in M(u,U)$ by hypothesis.
 \end{itemize}
We can then conclude that $w' \in uL(u,U)\cap vL(v,U)\neq \emptyset$. 
\\
Conversely, let 
  $uL(u,U)\cap vL(v,U)\neq \emptyset$. By Lemma \ref{lcmDentro}, we deduce that
 $w':=\mcm(u,v) =\frac{uv}{w}\in uL(u,U)\cap vL(v,U)$. 
 Then $w'=v \frac{u}{w}= vx_{j_1}^{\alpha_1}\cdots x_{j_l}^{\alpha_l} \in vL(v,U)$, so 
 $x_{j_1},...,x_{j_l}\in M(v,U)$  and, similarly,   
 $w'=u \frac{v}{w}=u x_{k_1}^{\beta_1}\cdots x_{k_s}^{\beta_s}\in uL(u,U)$, then  $x_{k_1},...,x_{k_s}\in M(u,U)$, allowing us to 
 conclude.
 \end{proof}

 Given $U:=\kT_D$, $D \in \NN$, suppose that $\forall u \in U$ a partition of the variables 
 into multiplicative and non-multiplicative ones
 $\{x_1,...,x_n\} =M(u,U) \sqcup NM(u,U)$
 is given, so that the binomial formula  (\ref{Binom}) holds and it never happens a situation as that described in Proposition  
\ref{Inibizione2}, then a relative involutive division is assigned, so $\cT(U)=\bigcup_{u \in U} C_L(u,U)$
and $\forall u,v \in U$, $C_L(u,U) \cap C_L(v,U) = \emptyset$.
\\
Indeed, the condition  $\forall u,v \in U$, $C_L(u,U) \cap C_L(v,U) = \emptyset$ immediately follows from 
Proposition  
\ref{Inibizione2}, whereas the condition  $\cT(U)=\bigcup_{u \in U} C_L(u,U)$ comes from the binomial formula 
(\ref{Binom}), observing that the relative  cones have been proved to be all disjoint. 

We have then found a criterion that, combined to condition (\ref{Binom}) allows us to define a relative involutive division.
\begin{example}\label{ConfigNoInscat}
Take, for example, $D=2$ and   $n=3$.
We have $U:=\mathcal{T}_2=\{x^2,xy,y^2,xz,yz,z^2\}$,
$\vert U\vert =6$.
\\
By Lemma \ref{PotPura}, we have to impose 
 $x \in \mult(x^2,U)$, 
$y \in \mult(y^2,U)$ and $z \in \mult(z^2,U)$. 

Then, by Corollary \ref{No2CnTutte}
 we have to take one and only one element of $U$ to which 
 assign all the variables as multiplicative; in our example 
 we take    $xy \in U$, so  
$\mult(xy,U)=\{x,y,z\}$.\\
By Proposition \ref{Inibizione2}, assigning 
$\mult(xy,U)=\{x,y,z\}$ imposes some limitations on 
the ways we can assign the multiplicative variables to the
 other terms.
\\
More precisely, since   $x \in \mult(xy,U)$ then 
 $y \notin \mult(x^2,U)$ and $y 
\in \mult(xy,U)$ implies $x \notin \mult(y^2,U)$.
\\
Moreover, $z \in \mult(xy,U)$ implies $y \notin \mult(xz,U)$, 
$x \notin \mult(yz,U)$ and  $\{x,y\}\nsubseteq \mult(z^2,U)$.
\\
After these first steps, we have the following configuration
\begin{center}
\begin{tabular}{|c|c|}
\hline
Terms & Multiplicative Variables\\
\hline
$x^2$& $x,\times ,?$\\
\hline
$xy$& $x,y,z$\\
\hline
$y^2$ &$\times ,y,?$\\
\hline
$xz$ & $?,\times ,?$\\
\hline
$yz$ & $\times ,?,?$\\
\hline
$z^2$ & $/,/z$\\
\hline
\end{tabular}
\end{center}
where we denote by the symbol 
 $?$ the free variables, i.e. those that can be freely assigned as 
 multiplicative for the term in the same row of the table\footnote{They are not affected 
 by the choice previously made.}, by  $\times$ 
  the variables which cannot be assigned as 
 multiplicative  and by $/$ the variables that cannot be contemporarily
  assigned as multiplicative for the term in the same row of the table.
 \\
 Now, taking into account the limitations stated above, we 
  assign two multiplicative variables to $xz \in U$:
$\mult(xz)=\{x,z\}$. 
\\
Again this choice will impose some limitations on 
the future choices:  $x \in \mult(xz,U)$ so $z \notin \mult(x^2)$, 
 whereas 
$z \in \mult(xz,U)$ implies $x \notin \mult(z^2,U)$.
\\
Then, after this step, we reach the configuration below: 
\begin{center}
\begin{tabular}{|c|c|}
\hline
Terms & Multiplicative Variables\\
\hline
$x^2$& $x,\times ,\times$\\
\hline
$xy$& $x,y,z$\\
\hline
$y^2$ &$\times ,y,?$\\
\hline
$xz$ & $x,\times ,z$\\
\hline
$yz$ & $\times ,?,?$\\
\hline
$z^2$ & $\times ,?,z$\\
\hline
\end{tabular}
\end{center}
We set now $\mult(z^2,U)=\{y,z\}$; 
  $y \in \mult(z^2,U)$ leads to  $z \notin \mult(y^2,U)$ and 
  $z \notin \mult(yz,U)$, so
\begin{center}
\begin{tabular}{|c|c|}
\hline
Terms & Multiplicative Variables\\
\hline
$x^2$& $x,\times ,\times$\\
\hline
$xy$& $x,y,z$\\
\hline
$y^2$ &$\times ,y,\times$\\
\hline
$xz$ & $x,\times ,z$\\
\hline
$yz$ & $\times ,?,\times$\\
\hline
$z^2$ & $\times ,y,z$\\
\hline
\end{tabular}
\end{center}
We get $\mult(y^2,U)=\{y\}$, which does not impose any relation 
on the future choices. 
Looking at the configuration, we finally get   $\mult(yz,U)=\{y\}$ and
$\mult(x^2)=\{x\}$, so we conclude with the configuration below:
\begin{center}
\begin{tabular}{|c|c|}
\hline
Terms & Multiplicative Variables\\
\hline
$x^2$& $x$\\
\hline
$xy$& $x,y,z$\\
\hline
$y^2$ &$y$\\
\hline
$xz$ & $x,z$\\
\hline
$yz$ & $y$\\
\hline
$z^2$ & $y,z$\\
\hline
\end{tabular}
\end{center}

 \end{example}

In Appendix \ref{elencone}, all the relative involutive divisions 
of $\kT_2 \subset \ck[x,y,z]$ are displayed, up to a permutation of the variables.
\medskip

% % % % % --> ritorno a Pommaret e provo che se le variabili sono inscatolate
% % % % %  allora esiste un riordino delle variabili per cui e' Pommaret.
\subsection{Pommaret relative involutive division}\label{PRID}
This section is devoted to the study of a complete characterization for 
  Pommaret relative involutive division.
  
    \begin{Proposition}\label{PommaretInscat}
    Let $U=\kT_D=\{u_1,...,u_l\}$, $D\in \NN$, $l={ D+n-1 \choose n-1}$ and suppose that a relative involutive division $L$
  is defined on $U$ in such a way that there is a relabelling of the terms in $U$ s.t.
  $$M_L(u_1,U) \subseteq M_L(u_2,U) \subseteq ... \subseteq M_L(u_l,U)=\{x_1,...,x_n\}.$$
  Let us denote for each $1 \leq i \leq n$, $\overline{u}_i=x_i^D$.
  Then, up to a reordering and a relabelling of the variables 
  the following properties hold:
  \begin{enumerate}
  \item $M(\overline{u}_1,U) \subseteq M_L(\overline{u}_2,U) \subseteq ... \subseteq M_L(\overline{u}_n,U) \subseteq\{x_1,...,x_n\}$,
   \item the (unique) term $u_l$ s.t. $M_L(u_l,U)=\{x_1,...,x_n\}$ has the form $u_l=x_i^D$, for some $1 \leq i \leq n$.
    \item for $1 \leq i \leq n$, $M_L(\overline{u}_i,U)=\{x_j \vert 1 \leq j\leq i \}.$
   \item under such reordering and relabelling of the variables, $L$ coincides with Pommaret division defined on $U$, i.e.
$$\forall u \in U, \, M_L(u,U)=\{x_i \vert x_i \leq \min(u)\}. $$
  \end{enumerate}
  \end{Proposition}
\begin{proof}
\noindent
 \begin{enumerate}
 \item For a particular reordering  %$x_{j_1}<x_{j_2}<...<x_{j_n}$ of the variables 
 we w.l.o.g. have
$$M_L(\overline{u}_{j_1},U) \subseteq M_L(\overline{u}_{j_2},U) \subseteq ... \subseteq M_L(\overline{u}_{j_n},U).$$
  It is then sufficient to choose a such reordering and to relabel the variables setting $x_i:=x_{j_i}$ for each $i$ in order to obtain the claim.
  \item 
 By Lemma \ref{PotPura}, for each $1 \leq i \leq n$, $x_i \in M(\overline{u}_i,U)$, so $x_1 \in M_L(\overline{u}_1,U), \{x_1,x_2\} \subseteq M_L(\overline{u}_2,U), ... \{x_1,...,x_n\}\subseteq M_L(\overline{u}_n,U)$
 so $M_L(\overline{u}_n,U)=\{x_1,...,x_n\}$, so $u_l=\overline{u}_n$ and then we can conclude.
  \item  In the proof of 2. we have already shown that, for each $1 \leq i \leq n$, 
 $M_L(\overline{u}_i,U)\supseteq \{x_j \vert 1 \leq j\leq i \}$, so we only need to prove that $\forall 1 \leq i \leq n$, $\forall i < j \leq n$,
 $x_j \notin M_L(\overline{u}_i,U)$.\\
 If $x_j \in M_L(\overline{u}_i,U)$, being $x_i \in M_L(\overline{u}_i,U) \subset M(\overline{u}_j,U)$, we have $\overline{u}_ix_j^D=u_jx_i^D  \in C_L(\overline{u}_i,U) \cap C_L(\overline{u}_j,U)$, leading
  to a contradiction.
  \item 
By 3., the inequalities stated in 4. are strict.\\
Taking the ordering $x_1<...<x_n$ on the variables we notice that, again by  3., the multiplicative variable sets
 for pure powers are in accordance with Pommaret division. Then, we only have to prove the assertion for non-pure powers.
 \smallskip
 
 Consider $u:=x_j^{\alpha_j}x_{j+1}^{\alpha_{j+1}} \cdots x_h^{\alpha_h} \in U$, $h>j$.
 We first prove that $\forall x_l>x_j$, $x_l \notin M_L(u,U)$.\\
Let $x_l \in _LM(u,U)$, $j< l $. %%abbiamo una variabile molt >min e proviamo che non possibile questo
\\
This implies $\{x_1,...,x_l\} \subset M(u,U)$; indeed, if for some $1 \leq k <l$, $x_k \notin M_L(u,U)$, 
then $\{x_1,...,x_k\}=M_L(x_k^D,U) \nsubseteq M_L(u,U)$ and $M_L(u,U) \nsubseteq M_L(x_k^D,U) $, contradicting the hypothesis.
\\
We take then $w:=\frac{ux_l}{x_j}=(x_j^{\alpha_j -1}    \cdots x_h^{\alpha_h})(x_l^{\alpha_l +1}) \in U$ 
and we can observe that $x_j \notin M_L(w,U)$, since $\frac{u}{GCD(u,w)}=x_j$ and 
$\frac{w}{GCD(u,w)}=x_l$ and we can apply Proposition \ref{Inibizione2} (since $x_l \in M(u,U)$, if $x_j \in M(w,U)$ then
$C_L(u,U) \cap C_L(w,U) \neq \emptyset$). 

Being $x_l>x_j$ and since the sets of multiplicative variables for the terms in $L$ form a chain of inclusions, we have 
$$x_j \notin M_L(w,U) \Rightarrow x_l \notin M_L(w,U).$$

Thus, consider $wx_l=\frac{ux_l^2}{x_j}=(x_j^{\alpha_j -1}    \cdots x_h^{\alpha_h} )(x_l^{\alpha_l +2}) $; 
we have $deg(wx_l)=D+1$, so it must belong to the relative involutive cone of one of its  predecessors, different 
from $w$, since $x_l \notin M_L(w,U)$. So take $x_m \neq x_l$ s.t. $x_m \mid wx_l$ and compute
$t:=\frac{wx_l}{x_m} \in U$. If $x_m \in M_L(t,U)$ then $\{x_1,...,x_m\} \subset M_L(t,U)$ and so
$tx_m x_j=ux_l^2 \in C_L(t,U) \cap C_L(u,U)$, which is impossible by definition of relative involutive division.
 
 Now, we only have to prove that $x_j=\min(u) \in M_L(u,U)$, because, if so, $\{x_1,...,x_j \}\subset M(u,U)$ and 
 by the argument above, this leads to the equality $\{x_1,...,x_j \}= M_L(u,U)$.
 
 Suppose then $x_j  \notin M_L(u,U)$ and consider $ux_j$. Since $\deg(ux_j)=D+1$, $ux_j$ must belong to the involutive cone of 
 one of its predecessors, but  it cannot belong to the cone of $u$. Let $x_m \mid ux_j$, $x_m \neq x_j$, but since $x_j=\min(u) $
 $x_m>x_j$. Consider $t=\frac{ux_j}{x_m}$; in order to have $ux_j \in C_L(t,U)$, $x_m$ should be multiplicative for $t$, 
 but this is impossible 
 since $\min(t)=x_j<x_m$.
 
 \end{enumerate}

\end{proof}

\begin{Corollary}  Let $U=\kT_D=\{u_1,...,u_l\}$, $D\in \NN$, $l={ D+n-1 \choose n-1}$ and suppose that a relative involutive division $L$
  is defined on $U$ in such a way that there is a relabelling of the terms in $U$ s.t.
  $$M_L(u_1,U) \subseteq M_L(u_2,U) \subseteq ... \subseteq M_L(u_l,U)=\{x_1,...,x_n\}.$$
  
  Then there is a reordering of the variables $x_{j_1}<x_{j_2}<...<x_{j_n}$ under which. $L$ coincides with Pommaret division defined on $U$, i.e.
$$\forall u \in U, \, M_L(u,U)=\{x_i \vert x_i \leq \min(u)\}. $$
 
\end{Corollary}

\section{The ideal and the graph}\label{graph}

% % % % % % % % % % % % % % % % % % % \textbf{QUI VA DISCUSSO COSA CI DEVO FARE: decomposizione sottoscala nella parte grado basso cosa devo farci per 
% % % % % % % % % % % % % % % % % % % esempio?????}
% % % % % % % % % % % % % % % % % % % 
% % % % % % % % % % % % % % % % % % % DOMANDA: decomposizione sottoscala grado <D magari dei coni scendono
In this section, given $U=\kT_D$, $D\in \NN$ and supposed that a relative involutive 
division $L$
 is defined on $U$, we deal with defining semigroup ideals generated by a subset of 
 terms in $U$ and the associated order ideals, in order to be consistent with the cone decomposition induced by $L$.
\\
In particular, given a term $t \in U$, if $t$ is a generator for the semigroup  ideal we want to construct,
all its multiple must belong to the same semigroup  ideal, so we have to consider as generators of the semigroup ideal all
the elements in $U$ whose involutive cone contain some multiple of $t$.
Clearly, an analogous argument apply to order ideals. Anyway, we first see the case of semigroup ideals.
\\
 Let $U=\kT_D$, $D\in \NN$  and suppose that a relative involutive division $L$
  is defined on $U$.\\
  From now on, for each $s,t \in U$, $X(s,t)$ will denote the unique element in $U$ s.t. 
  $\mcm(s, t)\in C_L(X(s,t),U) $; it exists by condition 1. (or 1'.) of Definition 
 \ref{RelInvDiv}, whereas its uniqueness is granted by condition 2. (or 2'.) of the same definition. 
\begin{Definition}\label{compliant}
With the above notation, a set $M \subseteq U$ is called \emph{compliant} if 
$\forall s,t  \in U$, $t \in M$ 
$ \Rightarrow X(s,t) \in M.$

% VECCHIA DEF $\forall s \in U, t \in M$
%   denoted by  $X(s,t)\in U \textrm{ the
%  term in } U \textrm{ s.t. } w:=\mcm(t,s) \in C_L(X(s,t),U)  $ it holds $ X(s,t)\in M. fINE VECCHIA DEF$
% % % % % % % % %   Let $U=\kT_D$, $D\in \NN$  and suppose that a relative involutive division $L$
% % % % % % % % %   is defined on $U$. A set $M \subseteq U$ is called \emph{compliant} if $\forall s \in U, t \in M$ 
% % % % % % % % %   denoted by  $ s'\in U \textrm{ the 
% % % % % % % % %  term in } U \textrm{ s.t. } w:=\mcm(t,s) \in C_L(s',U)  $ it holds $ s' \in M. $
\end{Definition}

\begin{Proposition} \label{IdealeGrafoFree}
 Let $U=\kT_D$, $D\in \NN$  and suppose that a relative involutive division $L$
  is defined on $U$. Let $M \subseteq U$ and $J=(M)$. Then it holds\\
$J=\bigcup_{l \in M} C_L(l,U) \Longleftrightarrow M $ is compliant.
\end{Proposition}

\begin{proof}
``$\Rightarrow$'' 
With the notation of Definition \ref{compliant}, assuming $t\in M$, we need to prove $X(s,t) \in M$.
% % 
% %  Consider $t,s \in U, t \in M $ and $w:=\mcm(t,s)$. 
% %  Let $s' \in U$ be the element in $U$ s.t. $w \in C_L(s'U)$; \textbf{questo va riformulato con la nuova nozione di completion} it exists by condition 1. (or 1'.) of Definition 
% %  \ref{RelInvDiv}, whereas its uniqueness is granted by condition 2. (or 2'.) of the same definition.
% %  We need to prove that $s'\in M$.\\
 Since $t\in M \then \mcm(s,t)\in J$ and  the only cone containing $\mcm(s,t)$ is $C_L(X(s,t),U)$, then $C_L(X(s,t),U)\subset J$ and hence $X(s,t)\in M$.
% % % % % %  Suppose by contradiction that  $s'\notin M$; since $t \mid w,$ then $w \in J=\bigcup_{l \in M} C_L(l,U)$ and so
% % % % % %  there is some $\hat{s} \in M$ s.t. $w \in C_L(\hat{s},U)\subset J$. Since $s' \notin M$ and $\hat{s} \in M$,
% % % % % %   then $s' \neq \hat{s}$, but this implies $w \in C_L(\hat{s},U)\cap C_L(s',U)\neq \emptyset$, which is impossible
% % % % % %    by definition of involutive division. Then $s'=\hat{s} \in M$.
 \\
 \medskip

  ``$\Leftarrow$'' We first observe that, clearly, $J \supseteq \bigcup_{l \in M} C_L(l,U)$ since the elements of this union are by definition multiples of some generators of $J$; then we have only to 
prove that $J \subseteq \bigcup_{l \in M} C_L(l,U)$, i.e. that 

$\forall u \in J$, $\exists v \in M$ s.t. $u \in C_L(v,U)$.\\

If $\deg(u)=D$, then $u \in M$, so $u \in C_L(u,U)$. If $\deg(U)=D+1$, then, since $U \in J$ and so it is multiple of some generator, there exists $v' \in M$ s.t. $u=v'x_j$ for some variable $x_j$.
If $x_j \in M_L(v',U)$, then $u \in C_L(v',U)$
Otherwise, we consider $x_i \mid v'$, $x_i \neq x_j$ \footnote{It is always possible to find  $x_i \mid v'$, $x_i \neq x_j$,
because otherwise $v'=X_j^D$ and so, by Lemma \ref{PotPura}  $x_j \in M_L(v',U)$.}. Clearly $x_i \mid u$, so we can take $v''=\frac{u x_j}{x_i}$; we have $v' \in M$, $v'' \in U$ and  
$\mcm(v',v'')=u$, so by hypothesis, $X(v',v'')\in M$, allowing us to conclude that $u \in \bigcup_{l \in M} C_L(l,U)$.
\\
\smallskip

In order to prove the claim we consider a counterexample of minimal degree.
So we assume to have
\begin{itemize}
 \item $u \in J$, $\deg(u) =D+h+1$, $h\geq 1$,
\item the single element $v \in U$ s.t. $u \in C_L(v,U)$ (whose existence and uniqueness are granted by definition of relative involutive division) and which satisfies $v \in \cN(J)$ thus giving the required contradiction;
\item $m \in L(v,U)$, $\deg(m)=h+1 : u=vm$;
\item since $u \in J$, it has a predecessor $w\in J, \deg(w)=D+h$ and 
\item a variable $x_j : u=x_jw$;
\item the element $w' \in M$ s.t. $w \in C_L(w',U)$ and
\item the related cofactor $m' \in L(w',U) : w=w'm'$
\end{itemize}
so that we have relation
\begin{equation}\label{tsm}
u=vm=wx_j=w'm'x_j.
\end{equation}
% % {\bf and the contradiction $u\in J\cap  C_L(v,U), v\notin M$
% % and that this is a counterexample of minimal value $h$, so that necessarily $h\geq 1$.} : {\bf non capisco il senso di mettere qui questa frase. La contraddizione l'avevi gia' detta prima e qui per me non c'entra piu'.}

We remark that
\begin{enumerate}
\item $v \in \cN(J), w' \in J \then v \neq w'$;
\item whence
%the minimality of the counterexample implies 
% [{\bf o il semplice fatto che i coni son disgiunti?}]
$w\notin C_L(v,U)$ whence
 \item the minimality of the counterexample implies $x_j \in NM_L(v,U)$ (otherwise $w$ would be a counterexample of lesser degree); 
 %[{\bf chi te lo assicura? Magari quella $x_j$ sta dentro $s$}]
 moreover    
\item $w\in C_L(w',U),\, u=wx_j\notin C_L(w',U) \then x_j \in NM_L(w',U)$ whence
\item $x_j \nmid m\in L(v,U)$ so that
\item $x_j \mid v$
\end{enumerate}
and we can define
\begin{itemize}
 \item $v'=\frac{v}{x_j},$
 \item $t_{v'}=\frac{v'}{\mcd(w',v')},$
 \item $t_m : m=t_{v'}t_m$
\end{itemize}
and get, dividing  \eqref{tsm} by $x_j$,

\begin{equation}
\frac{u}{x_j}=w=v'm=w't_{v'}t_m=w'm'.
\end{equation}
% 
% 
% \begin{equation} 
% \frac{u}{x_j}=w=v'm=w'm'=w't_{v'}t_m.
% \end{equation}
Then 
\begin{enumerate}\setcounter{enumi}{6}
 \item  $t_m= 1$ since, otherwise the term
$vt_{v'}$ would contradict minimality; so
\item $v'm=w' \frac{v'}{\mcd(w',v')}$ whence
\item $m=\frac{w'}{\mcd(w',v')} \mid w'$.
\item $\deg_j(\mcd(w',v') )= \min\{\deg_j(w'), \deg_j(v')\}$;
\item since $\deg_j(\mcd(w',v') )= \deg_j(v')\then x_j \mid m$ which is false, we have
\item $\deg_j(\mcd(w',v') )= \deg_j(w')$ and 
\item $\gcd(w',v')=\gcd(w',v)$.
\end{enumerate}
We have then  
 $$u=v'x_jm=vm=\frac{w'v}{\mcd(w',v')}=  \mcm(w',v) .$$
Now we have $w' \in M, v \in U, \mcm (w',v)=u \in C_L(v,U)$ which, by hypothesis implies $v=X(w',v) \in M$ proving that no counterexample exists.

\end{proof}

\begin{Remark}\label{cononecompliant}
 Let $U=\kT_D$ and suppose a relative involutive division $L$ to be defined on $U$. Let $t$ be the only
 element s.t. $M_L(t,U)=\{x_1,...,x_n\}$. The above proposition shows that if $M \neq \emptyset$ necessarily $t \in M$.
 Indeed, for each $u \in U$, $t=X(u,t)$.
\end{Remark}

  \begin{Proposition}\label{NoVerif}
  Let $U=\kT_D$, $D\in \NN$  and suppose that a relative involutive division $L$
  is defined on $U$. It holds $\forall s,t \in U$, $\exists u\in U$ s.t. $X(s,u)=t$ if and only if $X(s,t)=t$.
% % %   
% % %   
% % %   
% % %   Let $s,t \in U$;  there is a term $s' \in U$
% % %   such that $\mcm(s,s') \in C_L(t,U)$ if and only if $\mcm(t,s) \in C_L(t,U)$.
\end{Proposition}
\begin{proof}
 ``$\Leftarrow$''\\ It is trivial with $u=t$.
 \\
 ``$\Rightarrow$''\\ 
 Let $w=\mcm(s,u)$, then $w=tm$ with $m \in L(t,U)$.\\
 By definition of least common multiple $w=\frac{su}{\gcd(s,u)}$. We denote
 $h'=\gcd(s,u)$; so $h' \mid u$, $u=hh'$ and then $w=s \frac{u}{h'}=sh$.
 From $w=tm=sh$, we can deduce that $s=\frac{tm}{h}$.\\
 We compute now $l:=\mcm(t,s)=\mcm(t,\frac{tm}{h})$; clearly $t\mid l$, 
 but since $s=\frac{tm}{h}\mid tm$, then $l \mid tm$ so $l=tm'\mid tm$ and then $m'\mid m$
 so $m' \in L(t,U)$ and finally $l \in C(t,U).$
\end{proof}
With the above Proposition we can make simpler the verification that a set $M$ is 
compliant, giving an equivalent definition of compliance:
\begin{Definition}\label{Compliant2}
  Let $U=\kT_D$, $D\in \NN$  and suppose that a relative involutive division $L$
  is defined on $U$.\\
A set $M \subseteq U$ is called \emph{compliant} if 
$\forall s,t  \in U$ it holds
$$X(s, t)=t, \,s \in M \Rightarrow t \in M.$$

% % % RIFORMULARE 
% % % 
% % % X(s,t)=t e scrivere some nota che vuol dire che lcm(s,t) sta nel cono di t
\end{Definition}

\begin{example}\label{IdealPom}
 Consider now $n=3$, $D=2$ and define on $U=\kT_2$ the following division
 \begin{center}
\begin{tabular}{|c|c|}
\hline
Terms & Multiplicative Variables\\
\hline
$x^2$&  $x,y,z$\\
\hline
$xy$& $y,z$\\
\hline
$y^2$ &$y$\\
\hline
$xz$ & $z$\\
\hline
$yz$ & $y,z$\\
\hline
$z^2$ & $z$\\
\hline
\end{tabular}
\end{center}
%  If, with the notation of Proposition \ref{IdealeGrafoFree}, $xy\in M$, then 
%  $\mcm(x^2,xy)=x^2y \in C_L(x^2,U)$ so $x^2\in M$. The ideal 
%  $J=(x^2,xy)$ is exactly given by the union $C_L(x^2,U)\cup C_L(xy,U)$; indeed, 
%  all the multiples of $x^2$ belong to $C_L(x^2,U)$; the multiples of $xy$
%  involving only $y,z$ belong to $ C_L(xy,U)$, whereas all the multiples of $xy$ involving
%   $x$ are also multiples of $x^2$ so they belong to $C_L(x^2,U)$. 
If, with the notation of Proposition \ref{IdealeGrafoFree}, $xy\in M$, then
 $\mcm(x^2,xy)=x^2y \in C_L(x^2,U)$ so $x^2\in M$.

 This is the only $\mcm$ which must be computed ad thus the ideal
 $J=(x^2,xy)$ is exactly given by the union $C_L(x^2,U)\cup C_L(xy,U)$.

 Indeed,
 all the multiples of $x^2$ belong to $C_L(x^2,U)$ so
 $\frac{\mcm(x^2,t)}{x^2}=\frac{t}{\mcd(x^2,t)}\in L(x^2,U)$ and
$X(x^2,t)=x^2$ for all $t\in M$; the multiples of $xy$
 involving only $y,z$ belong to $C_L(xy,U)$, so
 $\frac{\mcm(xy,t)}{xy}=\frac{t}{\mcd(xy,t)}\in L(xy,U)$ and $X(xy,t)=xy$ for
each $t\in M\setminus\{x^2\}$;
 moreover all the multiples of $xy$ involving
  $x$ are also multiples of $x^2$ so they belong to $C_L(x^2,U)$.
\end{example}

\begin{example}\label{IdealDeg3}
 Consider $D=n=3$ and define the following relative involutive division:
 \begin{center}
\begin{tabular}{|c|c|}
\hline
Terms & Multiplicative Variables\\
\hline
$x^3$& $x,z$\\
\hline
$x^2y$& $x $\\
\hline
$xy^2$ &$x$\\
\hline
$y^3$ & $x,y$\\
\hline
$x^2z$ & $z$\\
\hline
$xyz$ & $x,y,z$\\
\hline
$y^2z$ &$y$\\
\hline
$xz^2$ & $z$\\
\hline
$yz^2$ & $y$\\
\hline
$z^3$ & $y,z$\\
\hline
\end{tabular}
\end{center}
 With the notation above, suppose $xy^2 \in M$, then:
 \begin{itemize}
  \item  $\mcm(xy^2,y^3)=xy^3 \in C_L(y^3,U) \Rightarrow y^3 \in M$
  \item  $\mcm(xy^2,xyz)=xy^2z \in C_L(xyz,U) \Rightarrow xyz \in M$
  \item  $\mcm(y^3,y^2z)=y^3z \in C_L(y^2z,U) \Rightarrow y^2z \in M$
   \item  $\mcm(yz^2,y^2z)=y^2z^2 \in C_L(yz^2,U) \Rightarrow yz^2 \in M$
   \item  $\mcm(yz^2,z^3)=yz^3 \in C_L(z^3,U) \Rightarrow z^3 \in M$
   \item  $\mcm(xz^2,z^3)=xz^3 \in C_L(xz^2,U) \Rightarrow xz^2 \in M$
   \item  $\mcm(xz^2,x^2z)=x^2z^2 \in C_L(x^2z,U) \Rightarrow x^2z \in M$
   \item  $\mcm(x^3,x^2z)=x^3z \in C_L(x^3,U) \Rightarrow x^3 \in M$
\item  $\mcm(x^3,x^2y)=x^3y \in C_L(x^2y,U) \Rightarrow x^2y \in M$
 \end{itemize}
So, $M=U$ and in this case clearly $J=\bigcup_{l \in M} C_L(l,U)$. 
\end{example}

We focus now on order ideals.

% % % % % % % % 
% % % % % % % % 1)per ogni $t\in M e s\in U definisco X(t,s)$
% % % % % % % % 2) definisco compliant come prima
% % % % % % % % 3) definisco come vuoi tu la proprietà
% % % % % % % % foreach $t\in M e s\in U, t=X(s,t)\then s\in M$
% % % % % % % % 
% % % % % % % % 
% % % % % % % % janet chiamava M le cose relative all'ideale e chiama N le cose relative all order ideal
% % % % % % % % Credo che sarebbe intelligente copiarlo

\begin{Definition}\label{revenant}
 Let $U=\kT_D$, $D \in \NN$. Suppose that a relative involutive division
$L$ is defined on $U$.
\\ A set $N \subseteq U$ is called \emph{revenant} if $\forall t,s \in U$,
%  $\forall t \in N, \forall s \in U$,
 $  X(t,s)=t$, $t \in N$ then $s \in N$.
\end{Definition}

\begin{Proposition}\label{EscalierGrafoFree}
Let $U=\kT_D$, $D \in \NN$. Suppose that a relative involutive division
$L$ is defined on $U$ and let $N \subseteq U$. It holds
$$H:=\bigl(\bigcup_{l \in N}C_L(l,U)\bigr)\cup \{v \in \kT\, \vert \, \deg(v) <D\} $$
is an order ideal, if and only if $N$ is revenant.
\end{Proposition}
\begin{proof}
``$\Rightarrow$'' Consider   
  $ t \in N,\, s \in U$, s.t. $ w:=\mcm(t,s) \in C_L(t,U) \subset H$.\\
  Since H is an order ideal and $s|w\in H$ then $s\in H$ whence $s\in N$.
% % % % % % %   If, by contradiction, $s \notin M$, then $s \notin H$; on the other
% % % % % % %   hand, $s \mid w\in H$, which leads to a contradiction, since $H$ is an order ideal.
  \\
``$\Leftarrow$''  
Let $w \in H$; we prove that every divisor of $w$ belongs to $H$ as well.
\\
If $\deg(w)\leq D$, its divisors have degree strictly smaller than $D$, so they clearly belong to $H$.\\
If $\deg(w)> D$, we know that $w \in C_L(t,U)$, for some $t \in N$, then we may write $w=tu$, with $u \in L(t,U)$.
Let $l \mid w$; if $t \mid l$, then $l \in C_L(t,U) \subset H$.
Otherwise, if $\deg(l)<D$, then $l \in H$ by definition of $H$, so the only case we have still to examine is the case $t \nmid l$ and
$\deg(l)\geq D$. By definition of relative involutive division, $\exists m\in U$ s.t. $l \in C_L(m,U)$, so $l=mv$, $v \in M(m,U)$. Take now
$w':=\mcm(m,t)=th'$ with $h'\mid u$, since $th'\mid w=tu$. 
 This implies
$\mcm(m,t)=w' \in C_L(t,U)$ and $X(m,t)=t$ so, by hypothesis, $m \in N$, allowing us to conclude that $l \in H$.
\end{proof}
\begin{example}\label{EscalierPom}
 Consider $n=3$, $D=2$ and the division defined in \ref{IdealPom}.
Suppose that $xz \in N$, then $\mcm(xz,z^2)=xz^2 \in C_L(xz,U)$, so $z^2 \in N$
and, for each $s \in U \setminus \{xz,z^2\}$, $\mcm(xz,s)\notin C_L(xz,U)$
 $\mcm(z^2,s)\notin C_L(z^2,U)$, so $N=\{xz,z^2\}$ and 
 $\bigl(\bigcup_{l \in N}C_L(l,U)\bigr)\cup \{v \in \kT\, \vert \, \deg(v) <2\}$ is an order ideal.
\end{example}

\section{Pommaret division and the Ufnarovsky-like graph}\label{P_UL}
In this section, we focus on the particular case of Pommaret division.
We begin by remarking that:
\begin{itemize}
  \item if the Pommaret relative involutive division
$L$ is defined on $\kT_D$, then $\forall u_1,u_2 \in U$, $M_L(u_1,U) \subseteq M_L(u_2,U)$ or
$M_L(u_2,U) \subseteq M_L(u_1,U)$.
  \item a relative involutive division $L$ defined on 
 $U=\kT_D=\{u_1,...,u_l\}$, $D\in \NN$, $l={ D+n-1 \choose n-1}$
coincide with Pommaret division w.r.t. some variable reordering if and only if  
there is a relabelling of the terms in $U$ s.t.   $M(u_1,U) \subseteq M(u_2,U) \subseteq ... \subseteq M(u_l,U)=\{x_1,...,x_n\}$ (see Proposition \ref{PommaretInscat}).
\end{itemize}
% % % % % \begin{Lemma}\label{EquivInscat}
% % % % % Let $U=\kT_D$, $D \in \NN$. Suppose that the Pommaret relative involutive division
% % % % % $L$ is defined on $U$, i.e. $\forall u \in U$, $M_L(u,U)=\{x_i \vert x_i \leq \min(u)\}$ for some variable reordering. Then $\forall u_1,u_2 \in U$, $M_L(u_1,U) \subseteq M_L(u_2,U)$ or
% % % % % $M_L(u_2,U) \subseteq M_L(u_1,U)$. 
% % % % % \end{Lemma}
\medskip
  
\noindent
The Pommaret case is rather peculiar and  we can construct semigroup ideals/ order ideals in a simple  way. 
To do so, we define the \emph{Ufnarovsky-like graph} $G_U$, i.e. an oriented graph, whose vertices are the elements of 
$U=\kT_D$, $D \in \NN$ and s.t., given $s,t \in U$,
there is an edge $t \xrightarrow{x_j} s$ if and only if $\exists x_j \in NM_L(s,U)$ s.t. 
$sx_j \in C_L(t,U)$.
\begin{example}\label{ULGraph}
 Consider again 
$D=2$, $n=3$, and suppose that  Pommaret division with $ x<y<z$ is defined on $U$.\\
The Ufnarovsky-like graph is:
  
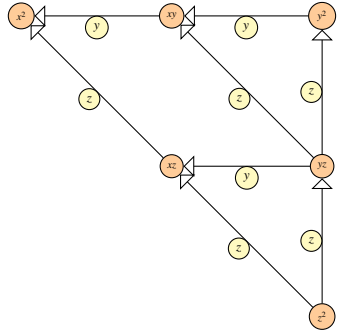
\begin{center}
\begin{tikzpicture}[every node/.style={circle, draw, scale=.5, fill=orange!40}, scale=1.0]
\node (2) at (2,0) [] {\footnotesize $x^2$};
\node (3) at (4,0) [] {\footnotesize $xy$};
\node (4) at (6,0) [] {\footnotesize $y^2$};
\node (6) at (4,-2) [] {\footnotesize $xz$};
\node (7) at (6,-2) [] {\footnotesize $yz$};
\node (9) at (6,-4) [] {\footnotesize $z^2$};

 \draw [-open triangle 90](3) -- node[auto, fill=yellow!30] {$y$} (2);
\draw [-open triangle 90](6) -- node[auto, fill=yellow!30] {$z$} (2);
\draw [-open triangle 90](4) -- node[auto, fill=yellow!30] {$y$} (3);
\draw [-open triangle 90](7) -- node[auto, fill=yellow!30] {$z$} (3);
 \draw [-open triangle 90](7) -- node[auto, fill=yellow!30] {$y$} (6);
 \draw [-open triangle 90](9) -- node[auto, fill=yellow!30] {$z$} (6);
  \draw [-open triangle 90](9) -- node[auto, fill=yellow!30] {$z$} (7);
  \draw [-open triangle 90](7) -- node[auto, fill=yellow!30] {$z$} (4);
\end{tikzpicture}
\end{center}

The variables labelling the edges, in the picture represent the non-multiplicative variabiles involved in drawing the Ufnarovsky-like graph; for example 

\begin{center}
\begin{tikzpicture}[every node/.style={circle, draw, scale=.5, fill=orange!40}, scale=1.0]

\node (2) at (2,0) [] {\footnotesize $x^2$};
\node (3) at (4,0) [] {\footnotesize $xy$};
 
 \draw [-open triangle 90](3) -- node[auto, fill=yellow!30] {$y$} (2);

\end{tikzpicture}
\end{center}
means thay $x^2\cdot y \in C_L(xy,U)$.
\end{example}

By means of the Ufnarovsky-like graph, we can construct both ideals and order ideals, as shown in 
what follows.

\begin{Definition}\label{ufcompl}
Let $U=\kT_D$, $D \in \NN$. Suppose that the Pommaret relative involutive division
$L$ is defined on $U$ and let $G_U$ the related Ufnarovsky-like graph.
\begin{itemize}
  \item   A subset $M\subset U$ is 
 \emph{Ufnarovsky-compliant} if
 $\forall s,t \in U$  for which there is an arrow
 $ t \rightarrow s \textrm{ in } G_U, \, s\in M \Rightarrow t \in M.$
  \item  Let $N \subset U$.\\
$N$ is \emph{Ufnarovsky-revenant} if
$\forall t,s \in U$ for which there is an arrow 
 $ t 
\rightarrow s \textrm{ in } G_U,\, t \in N \Rightarrow s \in N.$
\end{itemize}
%  
%   $\forall s \in M, \, \forall t \in U $  s.t. there is an arrow $ t \rightarrow s \textrm{ in } G_U \Rightarrow t \in M.$
\end{Definition}

\begin{Proposition}\label{PomIdeale}
Let $U=\kT_D$, $D \in \NN$. Suppose that the Pommaret relative involutive division
$L$ is defined on $U$ and let $G_U$ the related Ufnarovsky-like graph.

\begin{itemize}
  \item[a.] Let $M \subset U$ and $J=(M)$. Then
$J=(M)=\bigcup_{l \in M} C_L(l,U) \Leftrightarrow M$ is Ufnarovsky-compliant.
  \item[b.]  Let $N \subset U$. Then
$H:=\bigl(\bigcup_{l \in N}C_L(l,U)\bigr)\cup \{v \in \kT\, \vert \, \deg(v) <D\} \textrm{ is an order ideal}
\Leftrightarrow N$ is Ufnarovsky-revenant.
\end{itemize}
% % % % % % % % % % % \forall s \in M, \, \forall t \in U $  if there is an arrow $ t \rightarrow s \textrm{ in } G_U \Rightarrow t \in M.$
\end{Proposition}
\begin{proof}
  \begin{itemize}
    \item[a.] ``$\Rightarrow$''  \\ Consider $s \in M,\, t \in U$ s.t. there is an arrow $t \rightarrow s$; then
 $\exists x_j \in NM_L(s,U)$ s.t. 
$sx_j \in C_L(t,U)$.
\\
If,  by contradiction, $t \notin M$, being $x_j s \in J= 
\bigcup_{l \in M} C_L(l,U)$, $\exists w \in M$ (and so $w \neq t$) s.t. 
$x_j s \in C_L(w,U)$ and this contradicts condition 
$2'.$ of the definition of relative involutive division.\\
% 
% Suppose by contradiction that $t \notin M$. Since $s \in M$, then $x_j s \in J= 
% \bigcup_{l \in M} C_L(l,U)$, thus $\exists w \in M$ s.t. 
% $x_j s \in C_L(w,U)$. Being $t \notin M $ and $w \in M$, $t\neq w$ and 
% $x_j s \in C_L(t,U) \cap C_L(w,U) \neq \emptyset$, contradicting condition 
% $2'.$ of the definition of relative involutive division.\\
``$\Leftarrow$''
 Let $t \in U$, $s\in M$ and suppose $w=\mcm(s,t)\in C_L(t,U)$, so 
 $$w=\frac{st}{l}=th, \, l=\mcd(s,t),\, h=\frac{s}{l} \in L(t,U). $$
We have to prove that $t \in M$, by means of 
Ufnarovsky-compliance, i.e. by showing that there is
a path from $t$ to $s$ in $G_U$.\\ 
 Notice that $\deg(s)=\deg(t)$, so 
 $\deg(h)=\deg(\frac{s}{l})=\deg(\frac{t}{l})=:d$.
 We can write $h=\frac{s}{l}=y_1\cdots y_d$ and
 $\frac{t}{l}=z_1\cdots z_d$, with
 $y_i,z_i \in \{x_1,...,x_n\}$, $1 \leq i \leq d$,
 $y_1\leq ... \leq  y_d$, $z_1\leq ... \leq  z_d$.
 \\
 Since $ h \in L(t,U) $, we have $y_d=\max(h)\leq \min(t)$ and since $l \mid t$
  $y_d \leq \min(l)$. 
 \\
 Now take $s_1:=\frac{s z_1}{y_1} \in U$. 
 It holds 
$y_1\leq \min(t)\leq z_1$. If $y_1=z_1$, 
then $y_1\mid h=\frac{s}{l}$
and $y_1\mid \frac{t}{l}$ and this is impossible since
$\mcd(\frac{s}{l},\frac{t}{l})=1$
by definition of $l$. 
So we can conclude that $y_1< z_1$. We prove now that 
$y_1= \min(s_1).$\\
% In principle, $\min(s)\leq m_{h_1}$; but if $\min(s)< m_{h_1}$ then $\min(s) \mid l$, 
% $l \mid t$ so $\min(s)\mid t$
% and $\min(s) \nmid h$. 

In principle, $\min(s)\leq y_1$; but if $\min(s)< y_1
=\min(h)$ then $\min(s) \mid l$, $l \mid t$ 
so $\min(s)\mid t$ and $\min(s) \nmid h$,
being smaller than the minimal variable appearing in 
$h$ with nonzero exponent.

Then $\min(t) \leq \min(s) <y_1\leq y_d$ and 
this contradicts $h \in L(t,U),$
so $\min(s)=y_1$. 
Since, moreover,  $y_1< z_1$, 
then $y_1\leq \min(s_1)$, whence $y_1 \in M(s_1,U).$
From condition 2' of Definition \ref{DefRID} and 
observing that $s_1 y_1=sz_1$ we can desume that 
$s z_1\in C_L(s_1,U)$
\footnote{This could be seen also directly, 
since we have $\min(s)=y_1<z_1$} and 
this implies that have an arrow $s_1 \rightarrow s$ in
$G_U$ and so $s_1 \in M$.
\\
For $2 \leq j \leq d$, we define 
$$s_j:=\frac{s_{j-1} z_j}{y_j},$$
 with $z_j:=\min\Big(\frac{t}{l \prod_{i=1}^{j-1}z_i}\Big)=\min(z_j\cdots z_d)$
 and $y_j:=\min\Big(\frac{h}{\prod_{i=1}^{j-1}y_i}\Big)=\min(y_j\cdots y_d).$ 
\\ It holds $y_j\leq \max(h) \leq \min(t) \leq z_j$. \\
If $z_j=y_j$ then $y_j \mid \frac{t}{l},h$ but this 
is impossible since $\mcd\Big(\frac{t}{l},h\Big)=1$,
by definition of $l$,
so we have $z_j>y_j$.
\\
Now we prove that $y_j \leq \min(s_j)$. 
We know that $\prod_{i=1}^{j-1}y_i \mid s_{j-1},$ 
so $\min(s_{j-1})\leq y_j$.
If, by contradiction, $\min(s_{j-1})< y_j$, 
then $\min(s_{j-1})\nmid h$ and $\min(s_{j-1})\mid l$ 
and $l \mid t$, whence
$\min(t) \leq \min(s_{j-1}) < y_j \leq \max(h)$, 
which is impossible, so $\min(s_{j-1})= y_j< z_j$. 
Then $y_j\leq \min(s_j)$ and so $y_j\in M_L(s_j,U)$; 
$z_j>\min(s_{j-1})$ so $z_j\in NM_L(s_{j-1},U)$.
\\
By $s_jy_j=s_{j-1}z_j$, we get that there is an arrow
$s_j \rightarrow s_{j-1}$ in $G_U$, and so $s_j \in M$.\\
Finally, we observe that if $j=d$ we have $s_j=s_d=t$, 
so we can conclude. 
Indeed, we have found a path from $t\in U$ and $s \in M$,
so we can deduce that also $t \in M$.
    \item[b.]  ``$\Rightarrow $''\\
Consider $ t \in N,\, s \in U $ s.t. there is an arrow  $ t 
\rightarrow s \textrm{ in } G_U$; there is a variable $x_j \in NM_L(s,U)$ s.t. $sx_j =tv$, $v \in L(t,U)$.
In particular $\deg(sx_j)=D+1=\deg(tv)$ and $\deg(t)=D$, so\footnote{It is clear that $x_j\neq x_i$, since otherwise
$s=t$ and $NM_L(s,U)=NM_L(t,U) \ni x_j=x_i\in M_L(t,U)$} $v=x_i \neq x_j$. We may write
$sx_j=tx_i$ and so $x_j \mid t$, so $\mcm(s,t)=\mcm(\frac{tx_i}{x_j}, t)=tx_i\in C_L(t,U)$.
By Proposition \ref{EscalierGrafoFree} we can conclude that $s \in N$.
 \\
``$\Leftarrow$'' The construction is exactly the same as for part a. The only difference is that now $t \in N$ and $s \in U$, so, once constructed a path from $t$ to $s$, we can conclude that each $s_j$ in the path belongs to $N$ so $s \in N$.
  \end{itemize}

  \end{proof}

\begin{example}
 Referring to the graph of example \ref{ULGraph}, we see that if we  
 want to construct an ideal $J=(M)$ and we suppose $xz \in M$, then, following the graph, $yz,z^2 \in M$ and
 actually $J=(z^2,yz,xz)=C_L(z^2,U)\cup C_L(yz,U)\cup C_L(xz,U)$.
 On the other hands, if we want to construct an order ideal and we suppose $xz\in N$ then $x^2 \in N$ and
 $H=C_L(xz,U)\cup C_L(x^2,U)\cup \{v \in \kT\, \vert \, \deg(v) <2\} $
is actually an order ideal. 
\end{example}

It may seem that the Ufnarovsky-like graph may be enough to treat the general solution.
Indeed, in some small examples, as the following Example 
 \ref{GrafoStrano}, the Ufnarovsky-like graph solves the problem even if the relative involutive division $L$ defined is \emph{not} Pommaret division.

 \begin{example}\label{GrafoStrano}
Consider $D=2$ and $n=3$. There are exactly $6$ terms in $\kT_2$, i.e. $\mathcal{T}_2=\{x^2,xy,y^2,xz,yz,z^2\}$.
We take the relative involutive division given by
\begin{center}
\begin{tabular}{|c|c|}
\hline
Terms & Multiplicative Variables\\
\hline
$x^2$& $x$\\
\hline
$xy$& $x,y,z$\\
\hline
$y^2$ &$y$\\
\hline
$xz$ & $x,z$\\
\hline
$yz$ & $y$\\
\hline
$z^2$ & $y,z$\\
\hline
\end{tabular}
\end{center}
The corresponding Ufnarovsky-like graph is
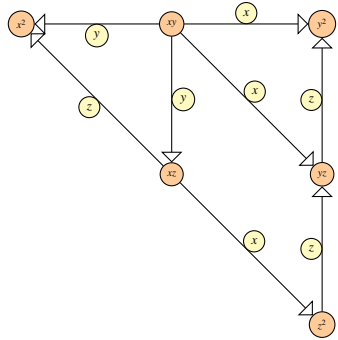
\begin{center}
\tikz[every node/.style={circle, draw, scale=.5, fill=orange!40}, scale=1.0]
{
% % % % % % \node (a) at (0,1) [circle] {$xy$};
% % % % % % \node (b) at (4.5,2) [circle] {$z^2$};
% % % % % % \node (c) at (1,4) [circle] {$xz$};
% % % % % % \node (d) at (7,2) [circle] {$x^2$};
% % % % % % \node (e) at (2,2) [circle] {$y^2 $};
% % % % % % \node (f) at (1.6,3) [circle] {$ yz$};
\node (d) at (2,0) [] {\footnotesize $x^2$};
\node (a) at (4,0) [] {\footnotesize $xy$};
\node (e) at (6,0) [] {\footnotesize $y^2$};
\node (c) at (4,-2) [] {\footnotesize $xz$};
\node (f) at (6,-2) [] {\footnotesize $yz$};
\node (b) at (6,-4) [] {\footnotesize $z^2$};
\draw [-open triangle 90](a) -- node[auto, fill=yellow!30] {$y$} (c);
\draw [-open triangle 90](a)-- node[auto, fill=yellow!30] {$y$} (d);
\draw [-open triangle 90](a) -- node[auto, fill=yellow!30] {$x$} (e);
\draw [-open triangle 90](a) -- node[auto, fill=yellow!30] {$x$}(f);
\draw [-open triangle 90](b) -- node[auto, fill=yellow!30]  {$z$}(f);
\draw [-open triangle 90](c) -- node[auto, fill=yellow!30] {$z$}(d);
\draw [-open triangle 90](c) -- node[auto, fill=yellow!30] {$x$}(b);
\draw [-open triangle 90](f) -- node[auto, fill=yellow!30] {$z$}(e);
}
  \end{center}
  Suppose we want to construct a monomial ideal $J=(M)$ s.t. $xz \in M$.
  Then, all multiples of $xz$ must belong to $J$. For those arising by multiplying 
  by  $x,z$ there is nothing to do, since these two variables are multiplicative
  for $xz$. For those involving $y$, we may add $xy$ to $M$ (indeed there is an arrow
  $xy \rightarrow xz$) and we can conclude since $M_L(xy,U)=\{x,y,z\}$.
  
 \end{example}
 \begin{Remark}\label{Supernota}
  In the paper \cite{PR}, Pleksen and Robertz employ the theory of Janet bases to compute resolutions of  \emph{ finitely generated modules over polynomial rings over fields and over rings of linear differential operators with coefficients in a differential field}.
Given the polynomial ring $R:=k[x_1,...,x_n]$ over a field $k$, they study the  free module $R^q$. First, they consider a multiple-closed set $S$ (i.e. a set of terms in $R^q$ which is closed by the multiplication for terms in $R$) and they take a Janet basis $J(S)$.
Then, they introduce as a tool the so-called Janet graph of $J(S)$.
The vertex set of the graph is given by $J(S)$ itself. For each $v \in J(S)$ and each $x_j$ that is non-multiplicative for $v$, let $w \in J(S)$ the unique involutive divisor of $vx_j$. Then, there is an edge $v\xrightarrow{x_j} w$.
\\
\smallskip

The paper \cite{S} is dedicated to \emph{structural properties of involutive divisions}. In particular, it extensively deals with Pommaret bases and its syzygy theory.
\\
In connection with the construction of a Groebner basis for the syzygy module of an involutive basis $H$, looking for a way to get an
involutive basis for the  module, the author applies Janet-Scheryer  theorem \cite{J1, Sch1}, which requires a suitable  ordering of the elements in $H$.\\
For producing it, the author generalizes the graph constructed by Pleksen and Robertz \cite{PR} to an arbitrary involutive division $L$, defining the $L$-graph.
This graph is defined exactly in the same way as the Janet graph, with vertex set $\cT(H)$, only the involutive division is different.
\\
From the $L$-graph, he defines the $L$-ordering on $H$, setting $h_a<h_b$ if there is a path from $\cT(h_a)$ to $\cT(h_b)$ in the graph.
\\
\medskip

In the paper \cite{HSS}, the authors study stable ideals, showing that they share many properties with the generic initial ideal.\\
Moreover, they relate Pommaret bases to some invariants associated with local cohomology and exhibit the existence of linear quotients in Pommaret bases. 
In such context, they take the Pommaret basis $H$ of a monomial ideal and they associate  Seiler's $L$-graph with $L$ given by Pommaret division.
This graph is again used to produce an ordering on $H$ (reversing that of \cite{S}).
\\
\smallskip

Finally, in \cite{AFS}, in the context of computing resolutions and Betti numbers, the authors again employ Janet-Schreyer theorem and the $J$-graph, i.e. essentially Pleksen-Robertz's Janet graph, with no significant modifications.
\\
\medskip

All these graphs (that are specifications of the most general version, i.e. that defined in \cite{S}) are very close to our Ufnarovsky-like graph. In particular, an $L$-graph is our graph with reversed arrows (and without the mark of the involved non-multiplicative variable, that we have put on each edge and that is only present in \cite{PR}). Moreover, we construct it on the set of  \emph{all} terms in some degree $D$ (so including the ideal and the escalier in the same graph), whereas they use the generators of a semigroup ideal. 

Our theory is restricted to the case of all monomials of the set $\kT_D\subset R=\ck[x_1,...,x_n]$, which have a fixed degree  $D$; in order to cover the theory of \cite{PR, S, HSS, AFS},  which  consider an involutive basis  $H=\{f_1,...,f_m\}\subset R^q$ and restrict to the leading terms, lying in $\kT_{\geq D}^{(m)}=\{te_i ,\, \deg(t)+\deg(f_i )\geq D\}$, we should extend our theory from $\kT_D$ to the $\kT_{\geq D}^{(m)}$.
 \end{Remark}

In the next section, we will see whether Ufnarovsky-like
 can cover the general case or not.
 
 \section{Relative involutive divisions, involutiveness and generic initial ideal }\label{Gin}
Riquier \cite{Riq}\footnote{This historical remark is deeply depending on \cite[IV.55]{SPES}} gave non only S-polynomial-like relations which must be satisfied by principal differential equation systems in order to have solutions but also described the initial conditions as series with the shape
$$\sum_{l\in N} l \sum{t\in C_L(l,U)} c_{l,t} t.$$

Delassus \cite{Del1,Del2,Del3} criticized Riquier for not having realized that his result was giving, up  to a generic change of coordinates, a \emph{canonical form} of the solutions. The (wrong) intuition by Delassus was the notion of \emph{generic initial ideal}. He considered\footnote{The notation used by these Hilbert's followers  is not obvious, the more so since Janet systematically reversed the notion of the other researchers; we present here Delassus claim using the present standard notations as described in \cite{GS}} the (degree)-lexicographical ordering $>$ induced by $x_1>\ldots>x_n$ and claimed that, given
$l$ independent forms $G:=\{g_1,\dots,g_l\}$ in $n$ variables of degree $p$, denoting $I:={\Bbb I}(G)$ the ideal generated by $G$ and ${\sf T}(G) = {\sf T}(I)$ its \emph{initial ideal},  there is a monomial ideal $gin(I)$ which not only satisfies $gin(I)={\sf T}(g(I))$ for all generic change of coordinate $g$\footnote{\emph{id est} for all $g\in U$, $U$ a Zariski open subset $U\in GL(n)$; of course that time was missing the notion  of Zariski sets, but the informal intuition was clear to researchers.} but even that it consisted of the first $l$ $>$-maximal monomials in ${\Cal T}_p$.

As it is well know, and as it was independently discovered by Gunther\cite{Gun1,Gun2,Gun3}
and Robinson \cite{Rob1,Rob2}, the result is at the same time false and not complete, but it
holds for all Borel-fixed monomial ideals.
%\footnote{The easiest counterexample being ${\Bbb I}(x_3^2,x_2x_3,x_2^2)$ 

Janet \cite{J2,J3,J4} in order to apply Cartan's test and condition applied himself a (Zariski open) generic change of coordinate obtaining the Pommaret involution and the related \emph{canonical form} which he called \emph{involutive}.

He stated that each non-trivial\footnote{if we consider the (trivially) Borel-fixed monomial 
ideal $U={\Cal T}_D$ of course the statement is trivially counterexampled by \emph{any} 
non-Pommaret relative involutive division.} Borel-fixed monomial
ideal is involutive \emph{id est} with our terminology its relative involutive division is Pommaret.

\begin{example}\label{JanetEx}\cite[p.31]{J3}
 Consider    $n=3$ and $D=2$, supposing $x<y<z$
and define Pommaret relative involutive division, as follows:
\begin{center}
\begin{tabular}{|c|c|}
\hline
Terms & Multiplicative Variables\\
\hline
$x^2$& $x$\\
\hline
$xy$& $x$\\
\hline
$y^2$ &$x,y$\\
\hline
$xz$ & $x$\\
\hline
$yz$ & $x,y$\\
\hline
$z^2$ & $x,y,z$\\
\hline
\end{tabular}
\end{center}
The associated Ufnarovsky-like graph is
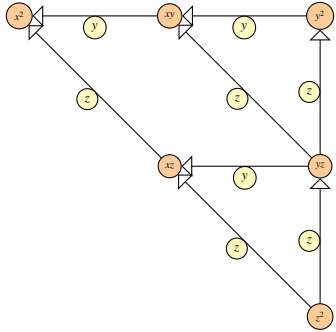
\begin{center}
\tikz[every node/.style={circle, draw, scale=.5, fill=orange!40}, scale=1.0]
{
\node (d) at (2,0) [] {\footnotesize $x^2$};
\node (a) at (4,0) [] {\footnotesize $xy$};
\node (e) at (6,0) [] {\footnotesize $y^2$};
\node (c) at (4,-2) [] {\footnotesize $xz$};
\node (f) at (6,-2) [] {\footnotesize $yz$};
\node (b) at (6,-4) [] {\footnotesize $z^2$};
\draw [-open triangle 90](a) -- node[auto, fill=yellow!30] {$y$} (d);
\draw [-open triangle 90](c) -- node[auto, fill=yellow!30] {$z$} (d);
\draw [-open triangle 90](e) -- node[auto, fill=yellow!30] {$y$} (a);
\draw [-open triangle 90](f) -- node[auto, fill=yellow!30] {$z$} (a);
\draw [-open triangle 90](f) -- node[auto, fill=yellow!30] {$z$} (e);
\draw [-open triangle 90](f) -- node[auto, fill=yellow!30] {$y$} (c);
\draw [-open triangle 90](b) -- node[auto, fill=yellow!30] {$z$} (c);
\draw [-open triangle 90](b) -- node[auto, fill=yellow!30] {$z$} (f);
}
  \end{center}
Suppose we want to construct an ideal and that $xy\in M$; looking at the above graph we see that this implies that also 
both  $y^2$ and  $yz $ must belong to $M$. Then also    $z^2\in M$ and we can conclude.
We have then $J = (xy,y^2,yz,z^2)$. Even it has been constructed using Pommaret division,
$J$ is not Borel-fixed \footnote{Indeed $xy \in J$ but $xz\notin J$.}, so it cannot be a gin.
 \end{example}

\section{A graph for the general solution}\label{GUL}

In section \ref{P_UL}, we have shown how to construct semigroup ideals and the corresponding escaliers in a consistent way w.r.t. the decomposition in cones induced by  Pommaret relative involutive divisions.
We do so by \emph{walking} in a simple graph, the Ufnarovsky-like graph.
\\
In this section, we will see that the Ufnarovsky-like graph cannot be used to cover the general case and we
will see how to generalize it, in order to cover all the possibile relative involutive divisions.
\\
Let us see an example
\begin{example}\label{Tr}
Take $D=3$ and $n=4$ so that we get the $20$ terms in
 $$\mathcal{T}_{3}=\{x^3,x^2y,xy^2,y^3,x^2z,xz^2,z^3,x^2t,xt^2,t^3,xyt,xzt,xyz, 
 yz^2,yt^2,y^2z,y^2t,z^2t,yzt,zt^2\}.$$
Let us then define the relative involutive division $L$, by specifying the 
multiplicative variables for each term:
\begin{center}
\begin{tabular}{|c|c|}
\hline
Terms & Multiplicative variables\\
\hline
$x^3$& $x$\\
\hline
$x^2y$& $x,y,t$\\
\hline
$xy^2$ &$y,z$\\
\hline
$y^3$ & $y,z$\\
\hline
$x^2z$ & $x,y,z$\\
\hline
$xz^2$ & $z$\\
\hline
$z^3$ & $z$\\
\hline
$x^2t$ & $x,z,t$\\
\hline
$xt^2$ & $y,t$\\
\hline
$t^3$ & $t$\\
\hline
$xyt$ & $y$\\
\hline
$xzt$ & $z,t$\\
\hline
$xyz$ & $z$\\
\hline
$yz^2$ & $z$\\
\hline
$yt^2$ & $t$\\
\hline
$y^2z$ & $z$\\
\hline
$y^2t$ & $y,t$\\
\hline
$z^2t$ & $z,t$\\
\hline
$yzt$ & $x,y, z,t$\\
\hline
$zt^2$ & $t$\\
\hline
\end{tabular}
\end{center}
and construct the corresponding Ufnarovsky-like graph:
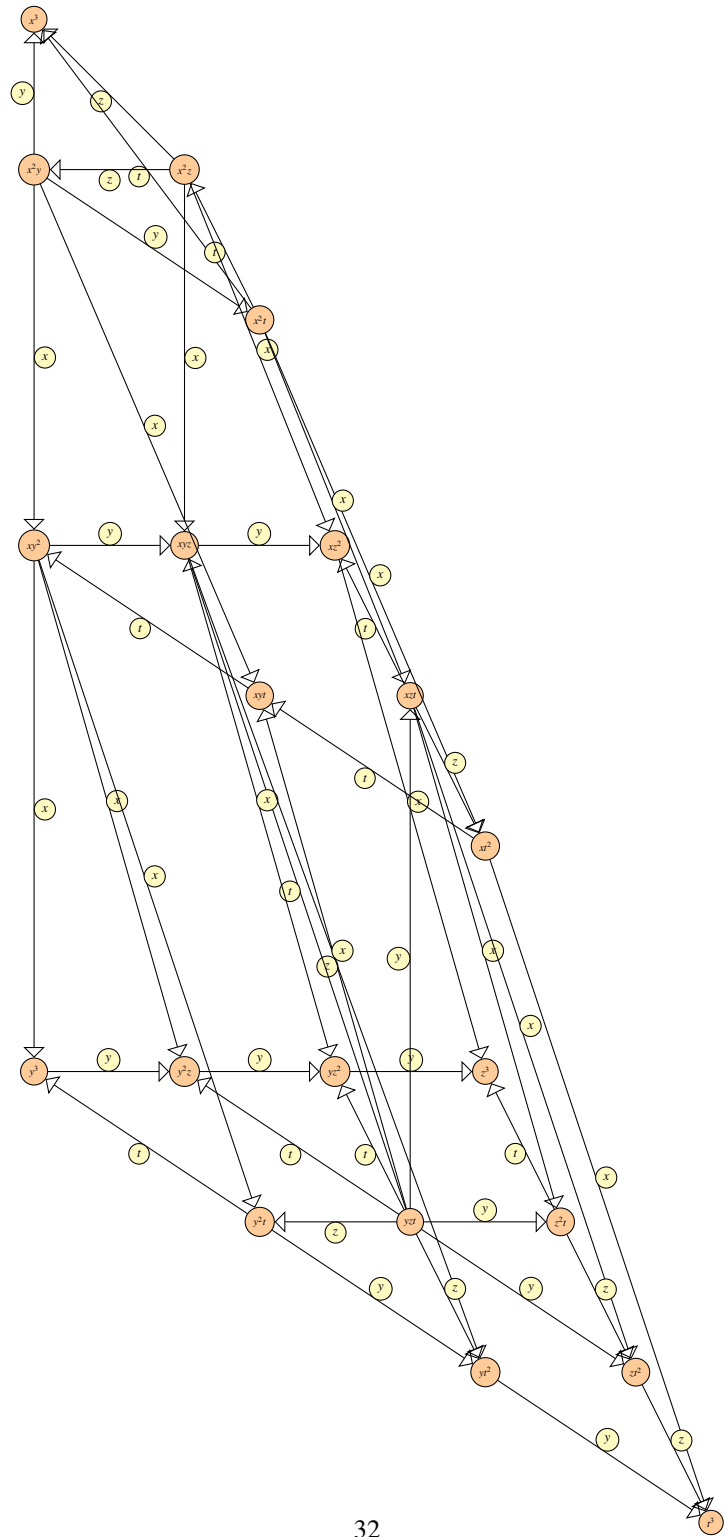
\begin{center}
\tikz[every node/.style={circle, draw, scale=.5, fill=orange!40}, scale=1.0]
{
% % % \node (a1) at (0,-2) [circle] {$x^3$};
% % % \node (a2) at (2,-2) [circle] {$x^2y$};
% % % \node (a3) at (10,3) [circle] {$xy^2$};
% % % \node (a4) at (12,1.5) [circle] {$y^3$};
% % % \node (a5) at (-2,0.5) [circle] {$x^2z $};
% % % \node (a6) at (1.6,1) [circle] {$xz^2$};
% % % \node (a7) at (2,2.5) [circle] {$z^3$};
% % % \node (a8) at (1,0.6) [circle] {$x^2t$};
% % % \node (a9) at (3,6) [circle] {$xt^2$};
% % % \node (a10) at (4,2) [circle] {$t^3$};
% % % \node (a11) at (6,4) [circle] {$xyt$};
% % % \node (a12) at (-1,3) [circle] {$ xzt$};
% % % \node (a13) at (3,-1) [circle] {$xyz$};
% % % \node (a14) at (7,6) [circle] {$yz^2$};
% % % \node (a15) at (8,-2) [circle] {$yt^2$};
% % % \node (a16) at (9,1) [circle] {$y^2z$};
% % % \node (a17) at (14,1) [circle] {$y^2t$};
% % % \node (a18) at (5,4) [circle] {$z^2t$};
% % % \node (a19) at (9.3,-0.5) [circle] {$yzt$};
% % % \node (a20) at (6.5,0) [circle] {$zt^2$};
\node (a1) at (-3,9) [] {\footnotesize $x^3$};
\node (a2) at (-3,7) [] {\footnotesize $x^2y$};
\node (a5) at (-1,7) [] {\footnotesize $x^2z$};
\node (a8) at (0,5) [] {\footnotesize $x^2t$};
\node (a3) at (-3,2) [] {\footnotesize $xy^2$};
\node (a13) at (-1,2) [] {\footnotesize $xyz$};
\node (a6) at (1,2) [] {\footnotesize $xz^2$};
\node (a11) at (0,0) [] {\footnotesize $xyt$};
\node (a12) at (2,0) [] {\footnotesize $xzt$};
\node (a9) at (3,-2) [] {\footnotesize $xt^2$};
\node (a4) at (-3,-5) [] {\footnotesize $y^3$};
\node (a16) at (-1,-5) [] {\footnotesize $y^2z$};
\node (a14) at (1,-5) [] {\footnotesize $yz^2$};
\node (a7) at (3,-5) [] {\footnotesize $z^3$};
\node (a17) at (0,-7) [] {\footnotesize $y^2t$};
\node (a19) at (2,-7) [] {\footnotesize $yzt$};
\node (a18) at (4,-7) [] {\footnotesize $z^2t$};
\node (a15) at (3,-9) [] {\footnotesize $yt^2$};
\node (a20) at (5,-9) [] {\footnotesize $zt^2$};
\node (a10) at (6,-11) [] {\footnotesize $t^3$};
\draw [-open triangle 90](a2) -- node[auto, fill=yellow!30] {$y$} (a1);
\draw [-open triangle 90](a5) -- node[auto, fill=yellow!30] {$z$} (a1);
\draw [-open triangle 90](a8) -- node[auto, fill=yellow!30] {$t$} (a1);
\draw [-open triangle 90](a5)-- node[auto, fill=yellow!30] {$z$} (a2);
\draw [-open triangle 90](a2)-- node[auto, fill=yellow!30] {$x$} (a3);
\draw [-open triangle 90](a11)-- node[auto, fill=yellow!30] {$t$} (a3);
\draw [-open triangle 90](a3)-- node[auto, fill=yellow!30] {$x$} (a4);
\draw [-open triangle 90](a17)-- node[auto, fill=yellow!30] {$t$} (a4);
\draw [-open triangle 90](a8)-- node[auto, fill=yellow!30] {$t$} (a5);
\draw [-open triangle 90](a5)-- node[auto, fill=yellow!30] {$x$} (a6);
\draw [-open triangle 90](a13)-- node[auto, fill=yellow!30] {$y$} (a6);
\draw [-open triangle 90](a12)-- node[auto, fill=yellow!30] {$t$} (a6);
\draw [-open triangle 90](a6)-- node[auto, fill=yellow!30] {$x$} (a7);
\draw [-open triangle 90](a14)-- node[auto, fill=yellow!30] {$y$} (a7);
\draw [-open triangle 90](a18)-- node[auto, fill=yellow!30] {$t$} (a7);
\draw [-open triangle 90](a2)-- node[auto, fill=yellow!30] {$y$} (a8);
\draw [-open triangle 90](a8)-- node[auto, fill=yellow!30] {$x$} (a9);
\draw [-open triangle 90](a12)-- node[auto, fill=yellow!30] {$z$} (a9);
\draw [-open triangle 90](a9)-- node[auto, fill=yellow!30] {$x$} (a10);
\draw [-open triangle 90](a15)-- node[auto, fill=yellow!30] {$y$} (a10);
\draw [-open triangle 90](a20)-- node[auto, fill=yellow!30] {$z$} (a10);
\draw [-open triangle 90](a2)-- node[auto, fill=yellow!30] {$x$} (a11);
\draw [-open triangle 90](a19)-- node[auto, fill=yellow!30] {$z$} (a11);
\draw [-open triangle 90](a9)-- node[auto, fill=yellow!30] {$t$} (a11);
\draw [-open triangle 90](a8)-- node[auto, fill=yellow!30] {$x$} (a12);
\draw [-open triangle 90](a19)-- node[auto, fill=yellow!30] {$y$} (a12);
\draw [-open triangle 90](a5)-- node[auto, fill=yellow!30] {$x$} (a13);
\draw [-open triangle 90](a3)-- node[auto, fill=yellow!30] {$y$} (a13);
\draw [-open triangle 90](a19)-- node[auto, fill=yellow!30] {$t$} (a13);
\draw [-open triangle 90](a13)-- node[auto, fill=yellow!30] {$x$} (a14);
\draw [-open triangle 90](a16)-- node[auto, fill=yellow!30] {$y$} (a14);
\draw [-open triangle 90](a19)-- node[auto, fill=yellow!30] {$t$} (a14);
\draw [-open triangle 90](a13)-- node[auto, fill=yellow!30] {$x$} (a15);
\draw [-open triangle 90](a17)-- node[auto, fill=yellow!30] {$y$} (a15);
\draw [-open triangle 90](a19)-- node[auto, fill=yellow!30] {$z$} (a15);
\draw [-open triangle 90](a3)-- node[auto, fill=yellow!30] {$x$} (a16);
\draw [-open triangle 90](a4)-- node[auto, fill=yellow!30] {$y$} (a16);
\draw [-open triangle 90](a19)-- node[auto, fill=yellow!30] {$t$} (a16);
\draw [-open triangle 90](a3)-- node[auto, fill=yellow!30] {$x$} (a17);
\draw [-open triangle 90](a19)-- node[auto, fill=yellow!30] {$z$} (a17);
\draw [-open triangle 90](a12)-- node[auto, fill=yellow!30] {$x$} (a18);
\draw [-open triangle 90](a19)-- node[auto, fill=yellow!30] {$y$} (a18);
\draw [-open triangle 90](a12)-- node[auto, fill=yellow!30] {$x$} (a20);
\draw [-open triangle 90](a19)-- node[auto, fill=yellow!30] {$y$} (a20);
\draw [-open triangle 90](a18)-- node[auto, fill=yellow!30] {$z$} (a20);
}
  \end{center}
This graph is rather complicated, 
% 
% METTERE LA MATRICE DI ADIACENZA
but we can isolate a part, to look more deeply 
into that one:
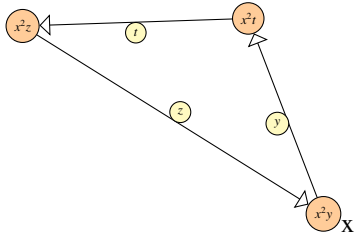
\begin{center}
\tikz[every node/.style={circle, draw, scale=.5, fill=orange!40}, scale=1.0]
{
\node (a2) at (2,-2) [circle] {$x^2y$};
\node (a5) at (-2,0.5) [circle] {$x^2z $};
\node (a8) at (1,0.6) [circle] {$x^2t$};
\draw [-open triangle 90](a5)-- node[auto, fill=yellow!30] {$z$} (a2);
\draw [-open triangle 90](a8)-- node[auto, fill=yellow!30] {$t$} (a5);
\draw [-open triangle 90](a2)-- node[auto, fill=yellow!30] {$y$} (a8);
}x
  \end{center}
  As usually, we denote $J=(M)$ a monomial ideal. If we want $x^2y \in M$, then we need $x^2z \in M$ since $(x^2y)z \in C(x^2z,U)$.
Now, if $x^2z \in M$, we must have $x^2t \in M$ and, for having $x^2t \in M$
we only need $x^2y\in M$, so, since each term has only the non-multiplicative variable we are using,  
it may seem that we can build an ideal considering only that three 
terms. This is actually impossible, since, for example $x^2yzt \notin C(x^2t) \cup C(x^2z) \cup C(x^2t)$  and, in any case we know that 
(see Remark \ref{cononecompliant}) $yzt\in M$.
\end{example}
\begin{Remark}\label{RmkApel}
 The paper \cite{A} provides a new approach to the theory of involutive divisions.
The given definition of involutive division is rather more general than that by Gerdt and Blinkov \cite{GB1, GB2}, but then the author restricts to  a smaller class  of involutive divisions, that he calls \emph{admissible}, giving effective criteria (somehow similar to our criteria for relative involutive divisions) for a division to be admissible.\\
Notice that an admissible division is not necessarily a relative involutive division, since (as for \cite{GB1, GB2}) the condition $\cT(U)=\bigcup_{u \in U} uL(u,U)$ is not required.
\\
The problem of constructing an involutive basis for an ideal\footnote{The paper is not involving also Groebner escalier's decomposition.}  is then tackledand the underlying idea is that for admissible involutive divisions it is enough to walk backwards in Ufnarovsky-like graph (but Apel does not introduce any graph as a tool).\\
It can be easily shown that the relative involutive division of example \ref{Tr} is not admissible. 
\end{Remark}

In order to cover the general case, we define a new oriented graph $G$, such that
 its nodes are the elements of $U$ and, given $t,s \in U$, there is an edge from $t$ to $s$
 if the following
  two conditions are verified:
  \begin{enumerate}
   \item there are no oriented paths from $t$ to $s$;
   \item  $\mcm(s,t) \in C_L(t,U)$
  \end{enumerate}
  
We call such a graph a \emph{generalized Ufnarovsky-like graph}.
 
  \begin{Definition}\label{gcompl}

    Let $U=\kT_D$, $D\in \NN$  and suppose that a relative involutive division $L$
  is defined on $U$.  Let $G$ be the generalized Ufnarovsky-like graph.

  \begin{enumerate}
   \item   A subset  $M\subseteq U$
  is \emph{G-compliant} if $ \forall t,s \in U$  for which in the oriented graph $G$ there is a
  path from $t$ to $s$, $s \in M$
$ \Rightarrow t \in M. $
   \item   A subset  $N\subseteq U$.
 is \emph{$G$-revenant} if   $\forall t,s \in U,$ for which  in the oriented graph $G$ there is a path from
$t$ to $s$, $ t \in N  \Rightarrow s \in N$. 

  \end{enumerate}
\end{Definition}

\begin{Proposition}\label{GrId}
   Let $U=\kT_D$, $D\in \NN$  and suppose that a relative involutive division $L$
  is defined on $U$.  Let $G$ be the generalized Ufnarovsky-like graph..
  
  \begin{itemize}
   \item[a.]   Let $M\subseteq U$, and $J=(M)$.
  $J=\bigcup_{l \in M} C_L(l,U) \Longleftrightarrow M $ is G-compliant.
  
   \item[b.] Let $N\subseteq U$, then 
 $H:=\bigl(\bigcup_{l \in N}C_L(l,U)\bigr)\cup \{v \in \kT\, \vert \, \deg(v) <D\} $
is an order ideal, if and only if   
$N$ is $G$-revenant.
  \end{itemize} 
  \end{Proposition}
\begin{proof}
\begin{itemize}
 \item[a.] ``$\Leftarrow$''\\ 
 Consider $s,s',t\in U$, $s \in M$ and suppose 
 $\mcm(s,s')\in C_L(t,U)$.
 By definition of $G$, there must be an edge from $t$ to $s$
 unless they are not already connected by a path, so by the hypothesis, we can conclude that
  $t \in M,$ whence $J=\bigcup_{l \in M} C_L(l,U)$ by Proposition \ref{IdealeGrafoFree}.
\\
``$\Rightarrow$''\\ 
Let $t,s \in U$, $s \in M$ and suppose there is a path from  $t$ to  
 $s$, i.e. $t \rightarrow s_r \rightarrow ... \rightarrow s_1 \rightarrow s$ is a path in $G$.
 \\
 Consider $s_1 \rightarrow s$; by definition of $G$, there is such an edge iff there is a term
 $s'' \in U$ s.t. $\lcm (s'',s) \in C_L(s_1,U)$. This implies (c.f. Proposition \ref{IdealeGrafoFree}) that $s_1 \in M$.
 Walking backwards this way in the path  $t \rightarrow s_r \rightarrow ... \rightarrow s_1
 \rightarrow s$ 
 of $G$, we get in a finite number of steps that $t \in M$.
 \item[b.] ``$\Rightarrow$''\\
 Let $t \in M$, $s \in U$ and suppose that 
 there is a path from  $t$ to  
 $s$, i.e. $t \rightarrow s_r \rightarrow ... \rightarrow s_1 \rightarrow s$ is a path in $G$.
 Since there is an edge $t \rightarrow s_r $ in $G$, then there is $s' \in U$ s.t.
 $\mcm(s_r,s')\in C_L(t,U)$ and, by Proposition \ref{NoVerif} 
 $\mcm(s_r,t)\in C_L(t,U)$ and so $s_r \in M$ by \ref{EscalierGrafoFree}. 
 \\
  Walking this way in the path  $t \rightarrow s_r \rightarrow ... \rightarrow s_1
 \rightarrow s$ 
 of $G$, e get in a finite number of steps that $s \in M$.\\
 ``$\Leftarrow$''\\
 Let us take $t \in M$, $s \in U$ and suppose $\mcm(t,s) \in C_L(t,U)$. This implies that
 in $G$ there is an edge $t \rightarrow s$ (or, at least, a path between them) and both
 the possibilities imply that $s \in M$. 
 \end{itemize}
  \end{proof}

  \begin{Remark}\label{Complex}
The rationale of imposing condition 1. to the definition of
 generalized Ufnarovsky-like graph is to have a minimal graph,
 with no redundant paths. The only condition 2. would make us construct
 a directed graph with many useless arrows.\\
 On the other hands, the only way we actually see to \emph{concretely construct} the generalized Ufnarovsky-like graph associated to a given relative involutive division $L$ is to construct the redundant graph using condition 2. and then eliminate the useless arrows.
 \\ 
  Looking at the above two propositions, we can notice that what we are interested is the path set 
   of the graph. In other words, condition 1. of its definition is not mandatory, but it is only set in order to
   avoid to consider useless edges.\\
   Considering a graph constructed using only condition 2., we may notice that the graph we defined is a 
   minimum equivalent digraph, i.e. the smallest subgraph preserving the path set.\\
   Finding a minimum equivalent digraph for a general directed graph is an NP-complete problem \cite{GGL}. 
   It would be hopeful to find the minimal graph in our particular cases with less effort, 
   or at least in some particular cases as we have already done for Pommaret division.

  \end{Remark}

\begin{example}\label{Facile}
Let us take $n=3$, $D=2$ and define the following division:
\begin{center}
\begin{tabular}{|c|c|}
\hline
Terms & Multiplicative Variables\\
\hline
$x^2$& $x$\\
\hline
$xy$& $x,y,z$\\
\hline
$y^2$ &$y,z$\\
\hline
$xz$ & $x,z$\\
\hline
$yz$ & $z$\\
\hline
$z^2$ & $z$\\
\hline
\end{tabular}
\end{center}
The corresponding graph is 

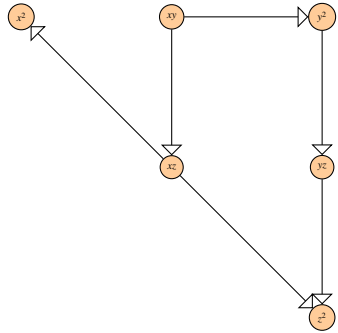
\begin{center}
\tikz[every node/.style={circle, draw, scale=.5, fill=orange!40}, scale=1.0]
{
\node (d) at (2,0) [] {\footnotesize $x^2$};
\node (a) at (4,0) [] {\footnotesize $xy$};
\node (e) at (6,0) [] {\footnotesize $y^2$};
\node (c) at (4,-2) [] {\footnotesize $xz$};
\node (f) at (6,-2) [] {\footnotesize $yz$};
\node (b) at (6,-4) [] {\footnotesize $z^2$};
\draw [-open triangle 90](a) --  (e);
\draw [-open triangle 90](a) -- (c);
\draw [-open triangle 90](e) --   (f);
\draw [-open triangle 90](f) --  (b);
\draw [-open triangle 90](c) --  (b);
\draw [-open triangle 90](c) -- (d);
}
  \end{center}
If $J=(M)$ and $xz\in M$, then $xy \in M$ and $J=(xz,xy)=C_L(xz,U)\cup C_L(xy,U)$.
 \\ On the same way, if we want to compute an order ideal and $xz \in N$ then both $x^2$ and $z^2$ 
 must be in $N$ and, with their three relative involutive cones (and, of course, 
 the elements of degree $0,1$), we have an order ideal.

\end{example}

\begin{example}\label{Tr2}
 Referring to the relative involutive division defined in example \ref{Tr}, we can see that, 
 if we restrict the graph only to the part involved in the argument introduced before, we have 
 
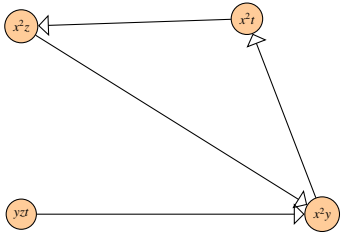
\begin{center}
\tikz[every node/.style={circle, draw, scale=.5, fill=orange!40}, scale=1.0]
{
\node (a2) at (2,-2) [circle] {$x^2y$};
\node (a5) at (-2,0.5) [circle] {$x^2z $};
\node (a8) at (1,0.6) [circle] {$x^2t$};
\node (BIG) at (-2,-2) [circle] {$yzt$};
\draw [-open triangle 90](a5)--  (a2);
\draw [-open triangle 90](a8)--   (a5);
\draw [-open triangle 90](a2)--  (a8);
\draw [-open triangle 90](BIG)--  (a2);
}
  \end{center}
  As usually, we denote $J=(M)$ a monomial ideal. If we want $x^2y \in M$, then we need $x^2z,yzt \in M$.
Now, if $x^2z \in M$, we must have $x^2t \in M$ and, for having $x^2t \in M$
we only need $x^2y\in M$.
  In this case $J=(x^2y,x^2z,x^2t,yzt)$ is exactly the union of the cones of its generators.
\end{example}

 \appendix
 \section{Relative involutive divisions}\label{elencone}
In this section, we summarize all the possible relative involutive divisions
for $U:=\kT_D$, $D=2$, $n=3$ up to symmetries.

We begin remarking that the orbits of $U$ under the symmetric group ${\sf S}_3$ are $\{x^2,y^2.z^2\}$ and $\{xy,yz.zx\}$.

It we choose $x^2$ as the ``source`` with $\{x,y,z\}$ as non-multiplicative variables, the simmetry group becomes $\langle(y,z)\rangle$ and
we need to fix one and only one term with 2 non-multiplicative variables among $\{xy,xz\}$; fixing $xy$, whose multiplicative variables are necessarily $\{y,z\}$ we have no more symmetries but we have still to set a second (and last) element with $\{y,z\}$ with multiplicative variables.
We can freely choose any among $\{y^2,yz,z^2\}$. Thus we obtain
\\
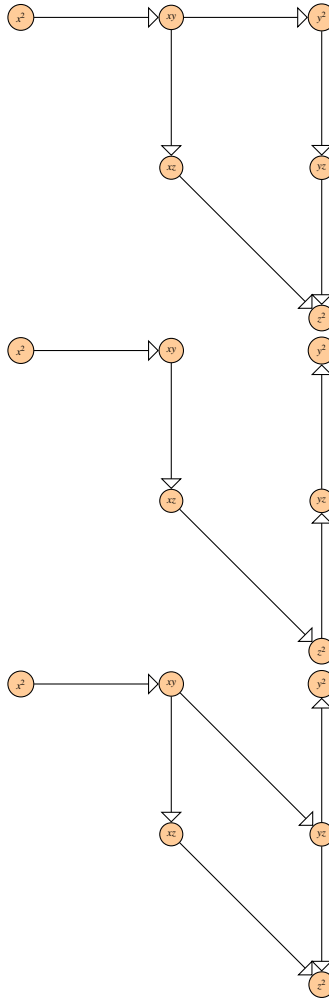
\begin{minipage}{4cm}
\begin{center}
\begin{tabular}{|c|c|}
\hline
Terms & Multiplicative Variables\\
\hline
$x^2$& $x,y,z$\\
\hline
$xy$& $y,z$\\
\hline
$y^2$ &$y,z$\\
\hline
$xz$ & $z$\\
\hline
$yz$ & $z$\\
\hline
$z^2$ & $z$\\
\hline
\end{tabular}
\end{center}
\end{minipage}
\hspace{2cm}
\begin{minipage}{4cm}
 
\begin{center}
\tikz[every node/.style={circle, draw, scale=.5, fill=orange!40}, scale=1.0]
{
\node (d) at (2,0) [] {\footnotesize $x^2$};
\node (a) at (4,0) [] {\footnotesize $xy$};
\node (e) at (6,0) [] {\footnotesize $y^2$};
\node (c) at (4,-2) [] {\footnotesize $xz$};
\node (f) at (6,-2) [] {\footnotesize $yz$};
\node (b) at (6,-4) [] {\footnotesize $z^2$};
\draw [-open triangle 90](a) --  (e);
\draw [-open triangle 90](a) -- (c);
\draw [-open triangle 90](e) --   (f);
\draw [-open triangle 90](f) --  (b);
\draw [-open triangle 90](c) --  (b);
\draw [-open triangle 90](d) -- (a);
}
\end{center}
\end{minipage}
\\
\begin{minipage}{4cm}
\begin{center}
\begin{tabular}{|c|c|}
\hline
Terms & Multiplicative Variables\\
\hline
$x^2$& $x,y,z$\\
\hline
$xy$& $y,z$\\
\hline
$y^2$ &$y$\\
\hline
$xz$ & $z$\\
\hline
$yz$ & $y$\\
\hline
$z^2$ & $y,z$\\
\hline
\end{tabular}
\end{center}
\end{minipage}
\hspace{2cm}
\begin{minipage}{4cm}
 
\begin{center}
\tikz[every node/.style={circle, draw, scale=.5, fill=orange!40}, scale=1.0]
{
\node (d) at (2,0) [] {\footnotesize $x^2$};
\node (a) at (4,0) [] {\footnotesize $xy$};
\node (e) at (6,0) [] {\footnotesize $y^2$};
\node (c) at (4,-2) [] {\footnotesize $xz$};
\node (f) at (6,-2) [] {\footnotesize $yz$};
\node (b) at (6,-4) [] {\footnotesize $z^2$};
% \draw [-open triangle 90](a) --  (e);
\draw [-open triangle 90](a) -- (c);
\draw [-open triangle 90](f) --   (e);
\draw [-open triangle 90](b) --  (f);
\draw [-open triangle 90](c) --  (b);
\draw [-open triangle 90](d) -- (a);
}
\end{center}
\end{minipage}
\\
\begin{minipage}{4cm}
\begin{center}
\begin{tabular}{|c|c|}
\hline
Terms & Multiplicative Variables\\
\hline
$x^2$&  $x,y,z$\\
\hline
$xy$& $y,z$\\
\hline
$y^2$ &$y$\\
\hline
$xz$ & $z$\\
\hline
$yz$ & $y,z$\\
\hline
$z^2$ & $z$\\
\hline
\end{tabular}
\end{center}
\end{minipage}
\hspace{2cm}
\begin{minipage}{4cm}
 
\begin{center}
\tikz[every node/.style={circle, draw, scale=.5, fill=orange!40}, scale=1.0]
{
\node (d) at (2,0) [] {\footnotesize $x^2$};
\node (a) at (4,0) [] {\footnotesize $xy$};
\node (e) at (6,0) [] {\footnotesize $y^2$};
\node (c) at (4,-2) [] {\footnotesize $xz$};
\node (f) at (6,-2) [] {\footnotesize $yz$};
\node (b) at (6,-4) [] {\footnotesize $z^2$};
 \draw [-open triangle 90](f) --  (e);
\draw [-open triangle 90](a) -- (c);
\draw [-open triangle 90](a) --   (f);
\draw [-open triangle 90](f) --  (b);
\draw [-open triangle 90](c) --  (b);
\draw [-open triangle 90](d) -- (a);
}
\end{center}
\end{minipage}
\\
\begin{minipage}{4cm}
% \begin{center}
% \begin{tabular}{|c|c|}
% \hline
% Terms & Multiplicative Variables\\
% \hline
% $x^2$& $x,y,z$\\
% \hline
% $xy$& $y,z$\\
% \hline
% $y^2$ &$y$\\
% \hline
% $xz$ & $z$\\
% \hline
% $yz$ & $y$\\
% \hline
% $z^2$ & $y,z$\\
% \hline
% \end{tabular}
% \end{center}
\end{minipage}
\hspace{2cm}

Next if we choose the ``source'' to be $xy$ we have again the simmetry group $\langle(x,y)\rangle$ and the orbits
$\{xy\}, \{xz,yz\},\{x^2,y^2\},\{z^2\}$.
We need to choose both an element in $\{x^2,xz,z^2\}$ with $\{x,z\}$ as non multiplicative variables and an element in
$\{y^2,yz,z^2\}$ with $\{y,z\}$ as non multiplicative variables;
note that the two choices are symmetric; so we can use the last freedom  for deciding to first choose the element in 
$\{x^2,xz,z^2\}$ as the one with $\{x,z\}$ as non multiplicative variables and next adapt the choice of the element with $\{y,z\}$ as non multiplicative variables consequently in $\{y^2,yz,z^2\}$; we can thus  either
\begin{itemize}
 \item choose respectively $x^2$ and $y^2$ obtaining\\
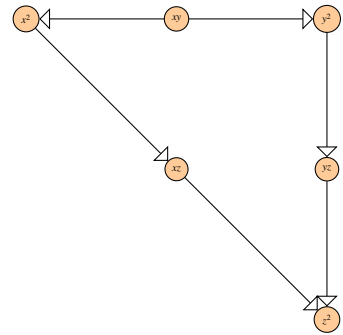
 \begin{minipage}{4cm}
\begin{center}
\begin{tabular}{|c|c|}
\hline
Terms & Multiplicative Variables\\
\hline
$x^2$& $x,z$\\
\hline
$xy$& $x,y,z$\\
\hline
$y^2$ &$y,z$\\
\hline
$xz$ & $z$\\
\hline
$yz$ & $z$\\
\hline
$z^2$ & $z$\\
\hline
\end{tabular}
\end{center}
\end{minipage}
\hspace{2cm}
\begin{minipage}{4cm}
 
\begin{center}
\tikz[every node/.style={circle, draw, scale=.5, fill=orange!40}, scale=1.0]
{
\node (d) at (2,0) [] {\footnotesize $x^2$};
\node (a) at (4,0) [] {\footnotesize $xy$};
\node (e) at (6,0) [] {\footnotesize $y^2$};
\node (c) at (4,-2) [] {\footnotesize $xz$};
\node (f) at (6,-2) [] {\footnotesize $yz$};
\node (b) at (6,-4) [] {\footnotesize $z^2$};
 \draw [-open triangle 90](a) --  (d);
\draw [-open triangle 90](a) -- (e);
\draw [-open triangle 90](e) --   (f);
\draw [-open triangle 90](d) --  (c);
\draw [-open triangle 90](c) --  (b);
\draw [-open triangle 90](f) -- (b);
}
\end{center}
\end{minipage}

\item or choose respectively $xz$ and $yz$  obtaining\\
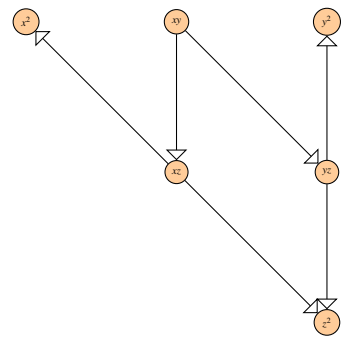
\begin{minipage}{4cm}
\begin{center}
\begin{tabular}{|c|c|}
\hline
Terms & Multiplicative Variables\\
\hline
$x^2$& $x$\\
\hline
$xy$& $x,y,z$\\
\hline
$y^2$ &$y$\\
\hline
$xz$ & $x,z$\\
\hline
$yz$ & $y,z$\\
\hline
$z^2$ & $z$\\
\hline
\end{tabular}
\end{center}
\end{minipage}
\hspace{2cm}
\begin{minipage}{4cm}
 
\begin{center}
\tikz[every node/.style={circle, draw, scale=.5, fill=orange!40}, scale=1.0]
{
\node (d) at (2,0) [] {\footnotesize $x^2$};
\node (a) at (4,0) [] {\footnotesize $xy$};
\node (e) at (6,0) [] {\footnotesize $y^2$};
\node (c) at (4,-2) [] {\footnotesize $xz$};
\node (f) at (6,-2) [] {\footnotesize $yz$};
\node (b) at (6,-4) [] {\footnotesize $z^2$};
 \draw [-open triangle 90](c) --  (d);
\draw [-open triangle 90](c) -- (b);
\draw [-open triangle 90](a) --   (f);
\draw [-open triangle 90](a) --  (c);
\draw [-open triangle 90](f) --  (e);
\draw [-open triangle 90](f) -- (b);
}
\end{center}
\end{minipage}

\end{itemize}
which of course preserves $\langle(x,y)\rangle$ as simmetry group; or
\begin{itemize}
\item choosing respectively $x^2$ and $yz$   obtaining\\
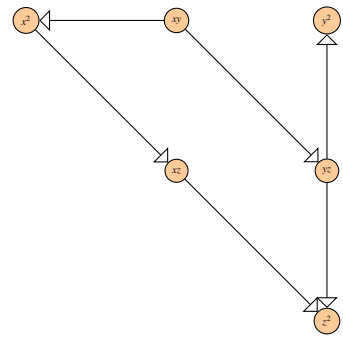
\begin{minipage}{4cm}
\begin{center}
\begin{tabular}{|c|c|}
\hline
Terms & Multiplicative Variables\\
\hline
$x^2$& $x,z$\\
\hline
$xy$& $x,y,z$\\
\hline
$y^2$ &$y$\\
\hline
$xz$ & $z$\\
\hline
$yz$ & $y,z$\\
\hline
$z^2$ & $z$\\
\hline
\end{tabular}
\end{center}
\end{minipage}
\hspace{2cm}
\begin{minipage}{4cm}
 
\begin{center}
\tikz[every node/.style={circle, draw, scale=.5, fill=orange!40}, scale=1.0]
{
\node (d) at (2,0) [] {\footnotesize $x^2$};
\node (a) at (4,0) [] {\footnotesize $xy$};
\node (e) at (6,0) [] {\footnotesize $y^2$};
\node (c) at (4,-2) [] {\footnotesize $xz$};
\node (f) at (6,-2) [] {\footnotesize $yz$};
\node (b) at (6,-4) [] {\footnotesize $z^2$};
 \draw [-open triangle 90](a) --  (d);
\draw [-open triangle 90](a) -- (f);
\draw [-open triangle 90](f) --   (b);
\draw [-open triangle 90](d) --  (c);
\draw [-open triangle 90](c) --  (b);
\draw [-open triangle 90](f) -- (e);
}
\end{center}
\end{minipage}

\item choose respectively $x^2$ and $z^2$ having no more freedom and obtaining\\
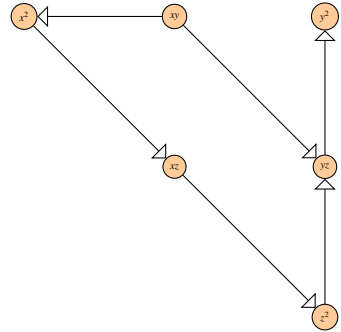
\begin{minipage}{4cm}
\begin{center}
\begin{tabular}{|c|c|}
\hline
Terms & Multiplicative Variables\\
\hline
$x^2$& $x,z$\\
\hline
$xy$& $x,y,z$\\
\hline
$y^2$ &$y$\\
\hline
$xz$ & $z$\\
\hline
$yz$ & $y$\\
\hline
$z^2$ & $y,z$\\
\hline
\end{tabular}
\end{center}
\end{minipage}
\hspace{2cm}
\begin{minipage}{4cm}
 
\begin{center}
\tikz[every node/.style={circle, draw, scale=.5, fill=orange!40}, scale=1.0]
{
\node (d) at (2,0) [] {\footnotesize $x^2$};
\node (a) at (4,0) [] {\footnotesize $xy$};
\node (e) at (6,0) [] {\footnotesize $y^2$};
\node (c) at (4,-2) [] {\footnotesize $xz$};
\node (f) at (6,-2) [] {\footnotesize $yz$};
\node (b) at (6,-4) [] {\footnotesize $z^2$};
 \draw [-open triangle 90](a) --  (d);
\draw [-open triangle 90](a) -- (f);
\draw [-open triangle 90](b) --   (f);
\draw [-open triangle 90](d) --  (c);
\draw [-open triangle 90](c) --  (b);
\draw [-open triangle 90](f) -- (e);
}
\end{center}
\end{minipage}

\item choose respectively $xz$ and $z^2$ having no more freedom and obtaining\\
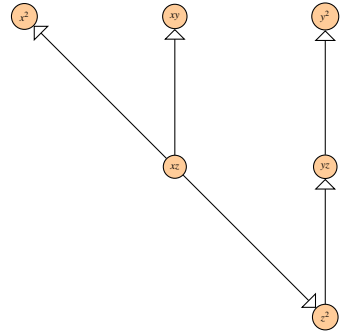
\begin{minipage}{4cm}
\begin{center}
\begin{tabular}{|c|c|}
\hline
Terms & Multiplicative Variables\\
\hline
$x^2$& $x$\\
\hline
$xy$& $x,y,z$\\
\hline
$y^2$ &$y$\\
\hline
$xz$ & $x,z$\\
\hline
$yz$ & $y$\\
\hline
$z^2$ & $y,z$\\
\hline
\end{tabular}
\end{center}
\end{minipage}\hspace{2cm}
\begin{minipage}{4cm}
 
\begin{center}
\tikz[every node/.style={circle, draw, scale=.5, fill=orange!40}, scale=1.0]
{
\node (d) at (2,0) [] {\footnotesize $x^2$};
\node (a) at (4,0) [] {\footnotesize $xy$};
\node (e) at (6,0) [] {\footnotesize $y^2$};
\node (c) at (4,-2) [] {\footnotesize $xz$};
\node (f) at (6,-2) [] {\footnotesize $yz$};
\node (b) at (6,-4) [] {\footnotesize $z^2$};
 \draw [-open triangle 90](c) --  (d);
\draw [-open triangle 90](c) -- (a);
\draw [-open triangle 90](b) --   (f);
% \draw [-open triangle 90](f) --  (a);
\draw [-open triangle 90](c) --  (b);
\draw [-open triangle 90](f) -- (e);
}
\end{center}
\end{minipage}
\end{itemize}

\section{Generalized Ufnarovsky-like graph for example \ref{Tr}.}

Example \ref{Tr}, shows that Ufnarovsky-like graph cannot be used for constructing
ideals and Groebner escaliers consistent with every relative involutive division. 
\\
In Section \ref{GUL}, we have defined the generalized Ufnarovsky-like graph and we proved that it solves
the problem in the general case.
\\
We show now the generalized Ufnarovsky-like graph for example \ref{Tr}.

\begin{center}
\tikz[every node/.style={circle, draw, scale=.5, fill=orange!40}, scale=1.0]
{
\node (a1) at (-3,9) [] {\footnotesize $x^3$};
\node (a2) at (-3,7) [] {\footnotesize $x^2y$};
\node (a5) at (-1,7) [] {\footnotesize $x^2z$};
\node (a8) at (0,5) [] {\footnotesize $x^2t$};
\node (a3) at (-3,2) [] {\footnotesize $xy^2$};
\node (a13) at (-1,2) [] {\footnotesize $xyz$};
\node (a6) at (1,2) [] {\footnotesize $xz^2$};
\node (a11) at (0,0) [] {\footnotesize $xyt$};
\node (a12) at (2,0) [] {\footnotesize $xzt$};
\node (a9) at (3,-2) [] {\footnotesize $xt^2$};
\node (a4) at (-3,-5) [] {\footnotesize $y^3$};
\node (a16) at (-1,-5) [] {\footnotesize $y^2z$};
\node (a14) at (1,-5) [] {\footnotesize $yz^2$};
\node (a7) at (3,-5) [] {\footnotesize $z^3$};
\node (a17) at (0,-7) [] {\footnotesize $y^2t$};
\node (a19) at (2,-7) [] {\footnotesize $yzt$};
\node (a18) at (4,-7) [] {\footnotesize $z^2t$};
\node (a15) at (3,-9) [] {\footnotesize $yt^2$};
\node (a20) at (5,-9) [] {\footnotesize $zt^2$};
\node (a10) at (6,-11) [] {\footnotesize $t^3$};
%  \draw [-open triangle 90](a2) --  (a1);
%  \draw [-open triangle 90](a5) --  (a1);
 \draw [-open triangle 90](a8) --  (a1);
%  \draw [-open triangle 90](a2)--  (a3);
%  \draw [-open triangle 90](a2)--  (a4);
 \draw [-open triangle 90](a5)--  (a2);
\draw [-open triangle 90](a2)--  (a8);
% \draw [-open triangle 90](a2)--  (a9);
% \draw [-open triangle 90](a2)--  (a10);
% \draw [-open triangle 90](a2)--  (a11);
% \draw [-open triangle 90](a2)--  (a15);
\draw [-open triangle 90](a2)--  (a17);
\draw [-open triangle 90](a3)--  (a4);
\draw [-open triangle 90](a5)--  (a3);
% \draw [-open triangle 90](a3)--  (a6);
% \draw [-open triangle 90](a3)--  (a7);
% \draw [-open triangle 90](a9)--  (a3);
\draw [-open triangle 90](a11)--  (a3);
\draw [-open triangle 90](a3)--  (a13);
% \draw [-open triangle 90](a3)--  (a14);
% \draw [-open triangle 90](a3)--  (a16);
% \draw [-open triangle 90](a5)--  (a4);
% \draw [-open triangle 90](a4)--  (a7);
% \draw [-open triangle 90](a9)--  (a4);
% \draw [-open triangle 90](a11)--  (a4);
\draw [-open triangle 90](a4)--  (a16);
% \draw [-open triangle 90](a4)--  (a14);
\draw [-open triangle 90](a17)--  (a4);
% \draw [-open triangle 90](a5)--  (a6);
% \draw [-open triangle 90](a5)--  (a7);
\draw [-open triangle 90](a8)--  (a5);
% \draw [-open triangle 90](a5)--  (a13);
% \draw [-open triangle 90](a5)--  (a16);
% % \draw [-open triangle 90](a5)--  (a14);
\draw [-open triangle 90](a6)--  (a7);
% \draw [-open triangle 90](a8)--  (a6);
% \draw [-open triangle 90](a12)--  (a6);
\draw [-open triangle 90](a13)--  (a6);
% \draw [-open triangle 90](a8)--  (a7);
% \draw [-open triangle 90](a12)--  (a7);
% \draw [-open triangle 90](a13)--  (a7);
\draw [-open triangle 90](a14)--  (a7);
% \draw [-open triangle 90](a16)--  (a7);
\draw [-open triangle 90](a18)--  (a7);
% \draw [-open triangle 90](a8)--  (a9);
% \draw [-open triangle 90](a8)--  (a10);
\draw [-open triangle 90](a8)--  (a12);
% \draw [-open triangle 90](a8)--  (a18);
% \draw [-open triangle 90](a8)--  (a20);
% \draw [-open triangle 90](a9)--  (a10);
\draw [-open triangle 90](a9)--  (a11);
\draw [-open triangle 90](a12)--  (a9);
% \draw [-open triangle 90](a9)--  (a15);
% \draw [-open triangle 90](a9)--  (a17);
% \draw [-open triangle 90](a12)--  (a10);
\draw [-open triangle 90](a15)--  (a10);
% \draw [-open triangle 90](a17)--  (a10);
% \draw [-open triangle 90](a18)--  (a10);
\draw [-open triangle 90](a20)--  (a10);
\draw [-open triangle 90](a11)--  (a17);
\draw [-open triangle 90](a12)--  (a18);
% \draw [-open triangle 90](a12)--  (a20);
\draw [-open triangle 90](a13)--  (a14);
\draw [-open triangle 90](a16)--  (a14);
\draw [-open triangle 90](a17)--  (a15);
\draw [-open triangle 90](a18)--  (a20);
\draw [-open triangle 90](a19)--  (a2);
}
  \end{center}


\begin{thebibliography}{bibb}
\bibitem{AFS} Albert, M., Fetzer, M.,  Seiler, W. M.,  \emph{Janet bases and resolutions in CoCoALib.} In International Workshop on Computer Algebra in Scientific Computing (pp. 15-29). Springer, Cham, 2015.

\bibitem{A} Apel, J., \emph{The theory of involutive divisions and an application to Hilbert function computations.} Journal of Symbolic Computation, 25(6), 683-704, 1998
% ==========================================================
\bibitem{Car1} 
E. Cartan
{\em Sur l'int\'egration des syst\`emes d'\'equations aux diff\'erentielles totals.}
Ann. \'Ec. Norm. $3^e$ s\'erie {\bf 18} (1901) 241.
% ==========================================================
\bibitem{Car2} 
E. Cartan
{\em Sur la structure des groupes infinis de transformations.}
Ann. \'Ec. Norm. $3^e$ s\'erie {\bf 21} (1904) 153.
% ==========================================================
\bibitem{Car3} 
E. Cartan
{\em Sur les syst\`emes en involution d'\'equations aux d\'eriv\'ees partielles du second ordre \`a une fonction inconnue de trois variables ind\'ependentes.}
Bull. Soc. Marth. {\bf 39} (1920) 356.
% ==========================================================
\bibitem{Del1}
Delassus E.,
{\em Extension du th\'eor\`eme de Cauchy aux syst\`emes les plus g\'en\'eraux d'\'equations 
aux d\'eriv\'ees partielles}. Ann. \'Ec. Norm. 
$3^e$ s\'erie {\bf 13} (1896) 421--467
% ==========================================================
\bibitem{Del2}
Delassus E.,
{\em Sur les syst\`emes alg\'ebriques et leurs relations avec certains 
syst\`emes d'equations aux d\'eriv\'ees partielles}. Ann. \'Ec. Norm. 
$3^e$ s\'erie {\bf 14} (1897) 21--44
% ==========================================================
\bibitem{Del3}
Delassus E.,
{\em Sur les invariants des syst\`emes diff\'erentiels}. Ann. \'Ec. Norm. 
$3^e$ s\'erie {\bf 25} (1908) 255--318
% ==========================================================
\bibitem{Ei} Eisenbud D., \emph{Commutative Algebra: with a view toward algebraic geometry}, Vol. 150. Springer Science $\&$ Business Media, 2013.
% ==========================================================
\bibitem{GAL} Galligo, A., \emph{ A propos du
th\'{e}orem de pr\'{e}paration de Weierstrass}, L.
N. Math.40 (1974), Springer, 543–579.
% ==========================================================

\bibitem{GGL} George, A. and Gilbert, J.R. and Liu, J.W.H., \emph{Graph Theory and Sparse Matrix Computation}, The IMA Volumes in 
Mathematics and its Applications, Springer New York (2012).

 \bibitem{GB1}
Gerdt V.P., Blinkov Y.A. 
{\em Involutive bases of Polynomial Ideals}, 
Math. Comp.  Simul.  {\bf 45} (1998),
543--560
   \bibitem{GB2}
Gerdt V.P., Blinkov Y.A. 
{\em Minimal involutive bases}, 
Math. Comp.  Simul.  {\bf 45} (1998),
519--541
\bibitem{GB3} 
Gerdt V.P., Blinkov Y.A. 
{\em Involutive Division Generated by an Antigraded Monomial Ordering}
L. N. Comp. Sci {\bf 6885} (2011), 
158-174, Springer
% ==========================================================
\bibitem{Gra} 
Grauert, H., \emph{{\"U}ber die Deformation isolierter Singularit\"aten analytischer Mengen}.  Inventiones mathematicae {\bf 15} (1971/72), 171-198
% ==========================================================
\bibitem{GS}
Green M.,  Stillman M., {\it A tutorial on generic initial ideals},  
in Buchberger B.,  Winkler F. (Eds.) {\it Gr\"obner Bases and Application} (1998)  
90--108  Cambridge Univ. Press
% ==========================================================
\bibitem{Gun1} Gunther, N., \emph{ Sur la forme canonique des syst\`emes d\'equations homog\`enes} (in russian)
[Journal de l'Institut des Ponts et Chauss\'ees de Russie]
Izdanie Inst. In\u z. Putej Soob\u s\u cenija Imp. Al. I. {\bf 84} (1913) .
% ==========================================================
\bibitem{Gun2}
Gunther, N.,   \emph{Sur la forme canonique des equations alg\'ebriques}
C.R. Acad. Sci. Paris {\bf 157} (1913),
577--80
% ==========================================================
   \bibitem{Gun3}
Gunther, N. 
{\em Sur les modules des formes alg\'ebriques} 
Trudy Tbilis. Mat. Inst. {\bf  9} (1941), 97--206
% ==========================================================
\bibitem{HSS} Hashemi, A., Schweinfurter, M., \& Seiler, W. \emph{Quasi-stability versus genericity.} In Computer Algebra in Scientific Computing (pp. 172-184). Springer Berlin/Heidelberg, (2012).
\bibitem{Hilbert}
Hilbert D.,
{\em Uber die Theorie der algebraicschen Formen}, 
Math. Ann. 
{\bf 36} (1890),
 473--534
\bibitem{Hir}
Hironaka, H. 
{\em Idealistic exponents of singularity} 
In: {\em Algebraic Geometry, The Johns Hopkins 	Centennial Lectures} (1977) 
52-125

 \bibitem{J1} M. Janet, \emph{Sur les syst\`{e}mes d'\'{e}quations aux d\'{e}riv\'{e}es partelles}, J. Math. Pure et Appl., $3$, $(1920)$, $65$-$151$.
\bibitem{J2} M. Janet, \emph{Les modules de formes alg\'{e}briques et la th\'{e}orie g\'{e}n\'{e}rale des systemes diff\'{e}rentiels}, Annales scientifiques de
l'\'{E}cole Normale Sup\'{e}rieure, 1924.
\bibitem{J3} M. Janet, \emph{Les  syst\`{e}mes d'\'{e}quations aux d\'{e}riv\'{e}es partelles}, Gauthier-Villars, 1927.
\bibitem{J4}, M. Janet, \emph{Lecons sur les syst\`{e}me$\sigma \in M$ tale ches d'\'{e}quations aux d\'{e}riv\'{e}es partelles }, Gauthier-Villars.

\bibitem{PR} Plesken, W., \& Robertz, D., \emph{Janet's approach to presentations and resolutions for polynomials and linear pdes}. Archiv der Mathematik, 84(1), 22-37, 2005.

  \bibitem{SPES}
T. Mora,  {\em Solving Polynomial Equation Systems} 4 Vols., Cambridge University
Press, I (2003), II (2005), III (2015), IV (2016).
\bibitem{Mum}
Mumford D., {\em Lectures on Curves on an Algebraic Surface} (1966) Princeton Univ. Press
% ==========================================================
\bibitem{Riq}
Riquier C.,
{\em Les syst\`emes d'\'equations aux d\'eriv\'ees  partielles}
(1910), Gauthiers-Villars.
% ==========================================================

\bibitem{Rob1}
Robinson, L.B.
{\em Sur les  syst\'emes d'\'equations aux d\'eriv\'ees partialles}
C.R. Acad. Sci. Paris {\bf 157} (1913),
106--108
% ==========================================================
\bibitem{Rob2}
Robinson, L.B.
{\em A new canonical form for systems of partial differential equations}
American Journal of Math. {\bf 39} (1917),
95--112
	\bibitem{Sch1} 
Schreyer  F.O.,
{\it Die Berechnung von Syzygien mit dem verallgemeinerten 
Weierstrass'schen Divisionsatz},
Diplomarbait, Hamburg (1980)
\bibitem{S} Seiler, W. M., \emph{ A combinatorial approach to involution and $\delta$-regularity II: Structure analysis of polynomial modules with Pommaret bases.}, Applicable Algebra in Engineering, Communication and Computing, 20(3), 261-338, 2009.
% ==========================================================
\bibitem{SeiB} Seiler, W.M., \emph{Involution: The formal theory of differential equations and its applications in computer algebra},
 Vol.24, 2009, Springer Science \& Business Media
% ==========================================================
\bibitem{U1} Ufnarovski V.,   
{\em A Growth Criterion for Graphs and Algebras Defined by Words}, 
Math. Notes {\bf 31} (1982), 238--241, 
% ==========================================================
\bibitem{U2} Ufnarovski V.,   
{\em On the use of Graphs for Computing a Basis, Growth and Hilbert Series of Associative Algebras}, 
Math. Sb. {\bf 180} (1989),1548--1560, 
% ==========================================================
\bibitem{U3} Ufnarovski V.,   
{\em Combinatorial and Asymptotic Methods in Algebra}, 
In: Kostrikin, A.I., Shafarevich, I.R.  (Eds.), 
{\it Algebra-VI} (1995)
Springer, 5--196  
% ==========================================================
\bibitem{U4} Ufnarovski V.,   
{\em Introduction to Noncommutative Gr\"obner Bases Theory}, 
In: Buchberger B.,  Winkler F. (Eds.), 
{\it Gr\"obner Bases and Application} (1998)
Cambridge Univ. Press, 259--280  



\bibitem{Vdm} 
Vandermonde, A.-T.
{\em Mémoire sur des irrationnelles de différens ordres avec une application au cercle} 
Histoire de l'Acad{\'e}mie royale des sciences (1775) 489--498 
\end{thebibliography}
\end{document}